\documentclass[a4paper,10pt]{amsart}
\usepackage[utf8]{inputenc}
\usepackage{amsxtra}
\usepackage{amsopn}
\usepackage{amsmath,amsthm,amssymb}
\usepackage{amscd}
\usepackage{amsfonts}
\usepackage{latexsym}
\usepackage{verbatim}
\usepackage{color}

\theoremstyle{plain}
\newtheorem{theorem}{Theorem}[section]
\newtheorem*{theorem*}{Theorem}

\newtheorem{lemma}[theorem]{Lemma}
\newtheorem{prop}[theorem]{Proposition}
\newtheorem{cor}[theorem]{Corollary}
\newtheorem{rem}[theorem]{Remark}
\newtheorem{ex}[theorem]{Example}
\newtheorem*{mt*}{Main Theorem}


\newcommand{\del}{\partial}
\newcommand{\delbar}{\bar{\del}}
\newcommand{\deltabar}{\bar{\delta}}
\DeclareMathOperator{\Imm}{Im}
\DeclareMathOperator{\Ker}{Ker}

\title[Differential operators on almost-Hermitian manifolds]{Differential operators on almost-Hermitian manifolds and harmonic forms}

\author[Nicoletta Tardini and Adriano Tomassini]{Nicoletta Tardini and Adriano Tomassini}

\address[Nicoletta Tardini]{Dipartimento di Matematica ``G. Peano'' \\
Universit\`{a} degli studi di Torino \\
Via Carlo Alberto 10\\
10123 Torino, Italy.
}
\email{nicoletta.tardini@gmail.com}
\email{nicoletta.tardini@unito.it}

\address[Adriano Tomassini]{Dipartimento di Scienze Matematiche, Fisiche e Informatiche\\
Unit\`{a} di Matematica e Informatica,
Universit\`{a} degli studi di Parma\\
Parco Area delle Scienze 53/A, 43124 \\
Parma, Italy}
\email{adriano.tomassini@unipr.it}

\thanks{\newline 
The first author is partially supported by SIR2014 project RBSI14DYEB ``Analytic aspects in complex and hypercomplex geometry'', and by GNSAGA of INdAM.
The second author is partially supported by the Project PRIN ``Varietà reali e complesse: geometria, topologia e analisi armonica'', Project PRIN 2017 ``Real and Complex Manifolds: Topology, Geometry and holomorphic dynamics''
and by GNSAGA of INdAM}
\begin{document}

\maketitle

\begin{abstract}
We consider several differential operators on compact almost-complex, almost-Hermitian and almost-K\"ahler manifolds. We discuss Hodge Theory for these operators and a possible cohomological interpretation. We compare the associated spaces of harmonic forms and cohomologies with the classical de Rham, Dolbeault, Bott-Chern and Aeppli cohomologies.
\end{abstract}

\section{Introduction}

On a complex manifold $X$ the exterior derivative $d$ decomposes as the sum of two other cohomological differential operators, namely $d=\del+\delbar$ satisfying
$\del^2=0$, $\delbar^2=0$ and $\del\delbar+\delbar\del=0$.
Once a Hermitian metric on $X$ is fixed one can associate to $\delbar$ a natural elliptic differential operator, the Dolbeault Laplacian; if $X$ is compact the kernel of this operator has a cohomological interpretation, i.e., it is isomorphic to the Dolbeault cohomology of $X$. If we do not assume the integrability of the almost-complex structure, i.e., $(X,\,J)$ is an almost-complex manifold, the $\delbar$ operator is still well-defined but it has no more a cohomological meaning. However, we can define some natural differential operators.\\
In this paper we are interested in studying the properties of such operators, their harmonic forms and possibly their cohomological meaning on compact manifolds endowed with a non-integrable almost-complex structure.
More precisely, in the non-integrable case $d$ decomposes as
$$
d:A^{p,q}(X)\to A^{p+2,q-1}(X)\oplus A^{p+1,q}(X)\oplus A^{p,q+1}(X)\oplus A^{p-1,q+2}(X)
$$
and we set
$$
d=\mu+\del+\delbar+\bar\mu\,.
$$
Then we define a $2$-parameter family of differential operators $\left\lbrace D_{a,b}\right\rbrace_{a,b\in\mathbb{C}\setminus\left\lbrace 0\right\rbrace}$ whose squares are zero and interpolate between $d$ and $d^c:=J^{-1}dJ$.
In general $d$ and $d^c$ do not anticommute and so in Proposition \ref{prop:anticommutation-relation} we give necessary and sufficient conditions on the parameters in order to have $D_{a,b}D_{c,e}+D_{c,e}D_{a,b}=0$; in such a case we define the Bott-Chern and Aeppli cohomology groups. Moreover, if we fix a $J$-Hermitian metric we develop a Hodge theory for these cohomologies together with the cohomology of $D_{a,b}$
(see Theorems \ref{thm:Dab-harmonic}, \ref{thm:Dab-harmonic-Dab-cohomology}, Proposition \ref{prop:Dab-star} and Theorems
\ref{thm:BCA-harmonic}, \ref{thm:BCA-harmonic-BCA-cohomology}). In particular we show that if $|a|=|b|$ then the cohomology of $D_{a,b}$ is isomorphic to the de Rham cohomology (cf. Proposition \ref{prop:isom-parametric-derham}). Moreover, in Example \ref{example:kodaira-thurston-Dab-cohomology} we compute explicitly the invariant $D_{a,b}$-cohomology on the Kodaira-Thurston manifold endowed with an almost-complex structure, showing that it is isomorphic to the de Rham cohomology independently on the parameters.
Nevertheless, the considered parametrized cohomology groups do not generalize (except for the almost-K\"ahler case) the classical Dolbeault, Bott-Chern and Aeppli cohomology groups of complex manifolds. To the purpose of finding a possible generalization of these cohomologies we consider the operators (cf. \cite{debartolomeis-tomassini})
$$
\delta:=\del+\bar\mu\qquad\bar\delta:=\delbar+\mu\,.
$$
These two operators anticommute but their squares are zero if and only if $J$ is integrable. In Section \ref{section:harmonic-forms-almost-hermitian} we define a generalization of the Dolbeault, Bott-Chern and Aeppli Laplacians and develop a Hodge theory for these operators studying their kernels.\\
In the almost-K\"ahler setting considered in Section \ref{section:harmonic-forms-almost-kahler} we derive some further relations among the kernels of these operators, involving also the Betti numbers and the dimension of $\bar\delta$-harmonic forms (see Corollary \ref{cor:comparison-betti-numbers}). A Hard-Lefschetz type Theorem for Bott-Chern harmonic forms is also proved (cf. Theorem
\ref{thm:hard-lefschetz}).\\
Finally, in the last Section we compute explicit examples on the two $4$-dimensional non-toral nilmanifolds and the Iwasawa manifold showing that a
bi-graded decomposition for the $\bar\delta$-harmonic forms cannot be expected and that the equalities in Theorem \ref{thm:equalities-harmonic-spaces} and the inequalities in Corollary \ref{cor:comparison-betti-numbers} are peculiar of the almost-K\"ahler case, giving therefore obstructions to the existence of a symplectic structure compatible with a fixed almost-complex structure on a compact manifold. In particular, we show in Example \ref{example-1} that even if in the bigraded case the spaces we consider coincide with the spaces considered in \cite{cirici-wilson-2}, this fails on total degree.

\medskip
\noindent{\sl Acknowledgments.}  The authors would like to thank Daniele Angella, Joana Cirici, Scott O. Wilson for interesting and useful discussions.
They also would like to thank the anonymous referee for reading carefully the paper and for providing useful comments that improved the presentation.
 Most of the work has been written during the first-named author's post-doctoral fellow at the Dipartimento di Matematica e Informatica ``Ulisse Dini'' of the Universit\`{a} di Firenze. She is grateful to the department for the hospitality.

\section{Preliminaries}\label{section:preliminaries}

Let $(X\,,J)$ be an almost-complex manifold then the almost-complex structure $J$ induces a natural bi-grading on the space of forms
$A^\bullet(X)=\bigoplus_{p+q=\bullet}A^{p,q}(X)$. If $J$ is non-integrable the exterior derivative $d$ acts on forms as
$$
d:A^{p,q}(X)\to A^{p+2,q-1}(X)\oplus A^{p+1,q}(X)\oplus A^{p,q+1}(X)\oplus A^{p-1,q+2}(X)
$$
and so it splits into four components
$$
d=\mu+\del+\delbar+\bar\mu\,,
$$
where $\mu$ and $\bar\mu$ are differential operators that are linear over functions. In particular, they are related to the Nijenhuis tensor $N_J$ by
$$
\left(\mu\alpha+\bar\mu\alpha\right)(X,Y)=\frac{1}{4} N_J(X,Y)
$$
where $\alpha\in A^1(X)$.
Since $d^2=0$ one has
$$
\left\lbrace
\begin{array}{lcl}
\mu^2 & =& 0\\
\mu\del+\del\mu & = & 0\\
\del^2+\mu\delbar+\delbar\mu & = & 0\\
\del\delbar+\delbar\del+\mu\bar\mu+\bar\mu\mu & = & 0\\
\delbar^2+\bar\mu\del+\del\bar\mu & = & 0\\
\bar\mu\delbar+\delbar\bar\mu & = & 0\\
\bar\mu^2 & =& 0
\end{array}
\right.\,
$$

Consider the following differential operators (cf. \cite{debartolomeis-tomassini})
$$
\delta:=\del+\bar\mu\,,\qquad \bar\delta:=\delbar+\mu
$$
with $\delta:A^{\pm}(X)\to A^\pm(X)$ and $\delta:A^{\pm}(X)\to A^\mp(X)$, where $A^\pm(X)$ are defined accordingly to the parity of $q$ in the $J$-induced bigraduation on $A^\bullet(X)$. 

\begin{lemma}
Let $(X,J)$ be an almost-complex manifold, the following relations hold
\begin{itemize}
\item[$\bullet$] $d=\delta+\bar\delta$,
\item[$\bullet$] $\delta^2+\bar\delta^2=0$,
\item[$\bullet$] $\delta^2=\del^2-\delbar^2$,
\item[$\bullet$] $\delta\bar\delta+\bar\delta\delta=0$.
\end{itemize}
\end{lemma}

\begin{proof}
The first statement follows immediately from the definitions. The second and third points follow from direct computation
$$
\bar\delta^2=(\delbar+\mu)(\delbar+\mu)=\delbar^2+\delbar\mu+\mu\delbar+\mu^2=\delbar^2-\del^2
$$
and, similarly, $\delta^2=\del^2-\delbar^2$.\\
Finally, for the last statement we have
$$
\delta\bar\delta+\bar\delta\delta=\del\delbar+\del\mu+\bar\mu\delbar+\bar\mu\mu+\delbar\del+\delbar\bar\mu+\mu\del+\mu\bar\mu=0.
$$
\end{proof}

If $D=d,\del,\delta,\bar\delta,\mu,\bar\mu$ we set $D^c:=J^{-1}DJ$, then
$\delta^c=-i\delta$ and $\bar\delta^c=i\bar\delta$ and
$$
d^c=i(\bar\delta-\delta)=i(\delbar+\mu-\del-\bar\mu).
$$
Notice that in general if $J$ is not integrable $d$ and $d^c$ do not anticommute, indeed we have
$$
dd^c+d^cd=2i(\bar\delta^2-\delta^2)=4i(\delbar^2-\del^2)\,.
$$
Therefore, an almost-complex structure $J$ is integrable if and only if $d^c=i(\delbar-\del)$ if and only if $d$ and $d^c$ anticommute.\\

Let $g$ be a $J$-Hermitian metric and denote with $*$ the associated anti-linear Hodge-*-operator. If $D=d,\del,\delbar,\mu,\bar\mu$ we set
$D^*:=-*D*$ and it turns out that $D^*$ is the adjoint of $D$ with respect to the $L^2$-pairing induced on forms (cf. \cite{debartolomeis-tomassini}, \cite{cirici-wilson-1}).\\
As usual one can consider the following differential operators
$$
\Delta_{\delbar}:=\delbar\delbar^*+\delbar^*\delbar\,,
$$
$$
\Delta_{\del}:=\del\del^*+\del^*\del\,,
$$
$$
\Delta_{\bar\mu}:=\bar\mu\bar\mu^*+\bar\mu^*\bar\mu\,,
$$
$$
\Delta_{\mu}:=\mu\mu^*+\mu^*\mu\,.
$$
While on compact almost-Hermitian manifolds the operators $\Delta_{\delbar}$,
$\Delta_{\del}$ are elliptic, and so the associated spaces 
$\mathcal{H}^{\bullet,\bullet}_{\delbar}(X):=\text{Ker}\,\Delta_{\delbar}$,
$\mathcal{H}^{\bullet,\bullet}_{\del}(X):=\text{Ker}\,\Delta_{\del}$ of harmonic forms are finite dimensional, in case of
$\Delta_{\bar\mu}$,
$\Delta_{\mu}$ the spaces
$\mathcal{H}^{\bullet,\bullet}_{\bar\mu}(X):=\text{Ker}\,\Delta_{\bar\mu}$ and 
$\mathcal{H}^{\bullet,\bullet}_{\mu}(X):=\text{Ker}\,\Delta_{\mu}$ are infinite-dimensional in general (recall that $\bar\mu$ and $\mu$ are linear over functions). In the following we will consider several spaces of harmonic forms and we will discuss the relations with these ones.

\section{Differential operators on almost-complex manifolds}

Let $(X,J)$ be an almost-complex manifold and consider a linear combination of the differential operators $\del\,,\delbar\,,\mu\,,\bar\mu$,
$$
D_{a,b,c,e}:=a\,\delbar+b\,\del+c\,\mu+e\,\bar\mu\,,
$$
with $a,b,c,e\in\mathbb{C}\setminus\left\lbrace 0\right\rbrace$.
Clearly $D_{a,b,c,e}$ satisfies the Leibniz rule; 
we are interested in finding conditions on the parameters ensuring that $D_{a,b,c,e}^2=0$.
Notice that if $J$ is integrable
$$
D_{a,b,c,e}:=a\,\delbar+b\,\del,
$$
and $D_{a,b,c,e}^2=0$ for any choice of the parameters.
Therefore from now on $J$ will always be assumed to be non-integrable if not stated otherwise.
In fact we have

\begin{lemma}
Let $(X,J)$ be an almost-complex manifold. Then $D_{a,b,c,e}^2=0$
if and only if
$$
e=\frac{a^2}{b}\quad\text{and}\quad c=\frac{b^2}{a}\,.
$$
\end{lemma}

\begin{proof}
By a direct computation one has
$$
D_{a,b,c,e}^2=\delbar^2(a^2-be)+\del^2(b^2-ac)+
(\del\delbar+\delbar\del)(ab-ce)\,.
$$
\end{proof}

We set
$$
D_{a,b}:=a\,\delbar+b\,\del+\frac{b^2}{a}\mu+\frac{a^2}{b}\bar\mu,
$$
with $a,b\in\mathbb{C}\setminus\left\lbrace 0\right\rbrace$.\\
Since $D_{a,b}^2=0$ we define the associated parametrized cohomology
$$
H_{D_{a,b}}^\bullet(X):=\frac{\text{Ker}\,D_{a,b}}{\text{Im}\,D_{a,b}}\,.
$$
Notice that if $a=b$, one has $D_{a,a}=a\,d$ i.e., a multiple of the exterior derivative.\\
In general, $D_{a,b}$ is not a real operator, indeed by a straightforward computation one gets

\begin{lemma}\label{lemma:real-operator}
Let $(X,J)$ be an almost-complex manifold. Then, $D_{a,b}=\overline{D_{a,b}}$
if and only if $a=\bar b$.
\end{lemma}
We set
$$
D_a:=a\,\delbar+\bar a\,\del+\frac{\bar a^2}{a}\mu+\frac{a^2}{\bar a}\bar\mu\,.
$$
Notice that the family of operators
$\left\lbrace D_a\right\rbrace_{a\in\mathbb{C}\setminus\left\lbrace 0\right\rbrace}$ contains the operators
$$
D_1=D_{1,1}=d\quad\text{and}\quad D_i=D_{i,-i}=d^c\,.
$$
In particular,
$$
H_{D_{1,1}}^\bullet(X)=H^\bullet_{dR}(X)\simeq
H_{d^c}^\bullet(X)=H_{D_{i,-i}}^\bullet(X)\,.
$$
Moreover, recall that if $J$ is non-integrable, $D_1D_i+D_iD_1\neq 0$, therefore we show when two real differential operators $D_a$ and $D_b$ anticommute.

\begin{prop}
Let $(X,J)$ be an almost-complex manifold. Then, $D_{a}D_b+D_bD_a=0$
if and only if $\bar a\,b\in\mathbb{R}$.
\end{prop}

\begin{proof}
Set
$$
D_a:=a\,\delbar+\bar a\,\del+\frac{\bar a^2}{a}\mu+\frac{a^2}{\bar a}\bar\mu
\quad\text{and}\quad
D_b:=b\,\delbar+\bar b\,\del+\frac{\bar b^2}{b}\mu+\frac{b^2}{\bar b}\bar\mu\,.
$$
Then, $D_{a}D_b+D_bD_a=0$ if and only if
$$
\left\lbrace
\begin{array}{lcl}
a\bar b+b\bar a & =& \frac{\bar a^2 b^2}{a\bar b}+
\frac{a^2\bar b^2}{\bar a b}\\[8pt]
2\bar a\bar b & =& \frac{\bar b^2a}{b}+\frac{\bar a^2b}{a}
\end{array}
\right.
$$
if and only if $\bar b a=\bar a b$ concluding the proof.
\end{proof}

In fact, with the same argument, more generally one has

\begin{prop}\label{prop:anticommutation-relation}
Let $(X,J)$ be an almost-complex manifold. Then, 
$D_{a,b}D_{c,e}+D_{c,e}D_{a,b}=0$
if and only if $ae=bc$.
\end{prop}

\begin{proof}
Set
$$
D_{a,b}:=a\,\delbar+b\,\del+\frac{b^2}{a}\mu+\frac{a^2}{b}\bar\mu
\quad\text{and}\quad
D_{c,e}:=c\,\delbar+e\,\del+\frac{e^2}{c}\mu+\frac{c^2}{e}\bar\mu.
$$
Then, $D_{a,b}D_{c,e}+D_{c,e}D_{a,b}=0$ if and only if
$$
\left\lbrace
\begin{array}{lcl}
ae+bc & =& \frac{ a^2 e^2}{bc}+
\frac{b^2c^2}{ae}\\[8pt]
2ac & =& \frac{bc^2}{e}+\frac{a^2e}{b}\\[8pt]
2be & =& \frac{ae^2}{c}+\frac{b^2c}{a}
\end{array}
\right.
$$
if and only if $ae=bc$ concluding the proof.
\end{proof}

\begin{rem}
Notice that when $J$ is integrable, it is straightforward to show that two arbitrary operators of the form
$$
D_{a,b}:=a\,\delbar+b\,\del
\quad\text{and}\quad
D_{c,e}:=c\,\delbar+e\,\del
$$
anticommute.
\end{rem}

\begin{rem}
If $b=1$, namely $D_b=d$ then
$$
D_ad+dD_a=0
$$
if and only if $a\in\mathbb{R}$. Namely, the only operators anticommuting with the exterior derivative in $\left\lbrace D_a\right\rbrace_{a\in\mathbb{C}\setminus\left\lbrace 0\right\rbrace}$ are those with the parameter $a$ real.\\
If $b=i$, namely $D_b=d^c$ then
$$
D_ad^c+d^cD_a=0
$$
if and only if $i\bar a\in\mathbb{R}$. Namely, the only operators anticommuting with $d^c$ in $\left\lbrace D_a\right\rbrace_{a\in\mathbb{C}\setminus\left\lbrace 0\right\rbrace}$ are those with the parameter $a$ purely imaginary.
\end{rem}

As a consequence of the previous considerations, if $ae=bc$ and $(a,b)\neq (c,e)$ then $(A^\bullet(X),D_{a,b},D_{c,e})$ is a double complex since
$$
\left\lbrace
\begin{array}{lcl}
D_{a,b}^2 & = & 0\\
D_{c,e}^2 & = & 0\\
D_{a,b}D_{c,e}+D_{c,e}D_{a,b} & = & 0
\end{array}
\right.,
$$
hence one can define the \emph{Bott-Chern} and 
\emph{Aeppli cohomologies} respectively as
$$
H_{BC(D_{a,b},D_{c,e})}^\bullet(X):=
\frac{\text{Ker}\,D_{a,b}\cap\text{Ker}\,D_{c,e}}
{\text{Im}\,D_{a,b}D_{c,e}},\quad
H_{A(D_{a,b},D_{c,e})}^\bullet(X):=
\frac{\text{Ker}\,D_{a,b}D_{c,e}}
{\text{Im}\,D_{a,b}+\text{Im}\,D_{c,e}}.
$$

\medskip

Let $(X,J)$ be an almost-complex manifold and let $g$ be a $J$-Hermitian metric on $X$. Then the adjoint of $D_{a,b}$ is
$$
D^*_{a,b}:=\bar a\,\delbar^*+
\bar b\,\del^*+\frac{\bar b^2}{\bar a}\mu^*+
\frac{\bar a^2}{\bar b}\bar\mu^*\,.
$$
We consider the second-order differential operator
$$
\Delta_{a,b}:=
D_{a,b}D^*_{a,b}+D^*_{a,b}D_{a,b}\,.
$$

\begin{lemma}
Let $(X,J)$ be an almost-complex manifold. The differential operators $D_{a,b}$ are elliptic.
\end{lemma}

\begin{proof}
Fix $a$ and $b$.
We can compute the symbol of $\Delta_{a,b}$ as follows.
We work in a local unitary frame of $T^*X$ and choose a basis 
$\left\lbrace\theta^1,\cdots,\theta^n\right\rbrace$ such that the metric can be written as 
$$
g=\theta^i\otimes\bar\theta^i+\bar\theta^i\otimes\theta^i\,.
$$
Using Einstein notations, a $(p,q)$-form $\alpha$ locally can be written as
$$
\alpha=\alpha_{i_1\cdots i_pj_1\cdots j_q}\theta^{i_1}\wedge\cdots
\theta^{i_p}\wedge\bar\theta^{j_1}\wedge\cdots\wedge\bar\theta^{j_q}\,.
$$
Then $\delbar$ acts as
$$
(\delbar\alpha)_{p,q+1}=\delbar_{j_{q+1}}
\alpha_{i_1\cdots i_pj_1\cdots j_q}\bar\theta^{j_{q+1}}\wedge\theta^{i_1}\wedge\cdots
\theta^{i_p}\wedge\bar\theta^{j_1}\wedge\cdots\wedge\bar\theta^{j_q}\,.
$$
and $\mu$ acts as
$$
\mu\alpha=
\alpha_{i_1\cdots i_pj_1\cdots j_q}\mu\left(\theta^{i_1}\wedge\cdots
\theta^{i_p}\wedge\bar\theta^{j_1}\wedge\cdots\wedge\bar\theta^{j_q}\right)
$$
and similarly for $\del$ and $\bar\mu$.
In computing the symbol of $\Delta_{a,b}$ we are only interested in the highest-order differential acting on the coefficients
$\alpha_{i_1\cdots i_pj_1\cdots j_q}$. 
Denoting with $\simeq$ the equivalence of the symbol of the operators we get
$$
\Delta_{a,b}\simeq
\vert a\vert^2\Delta_{\delbar}+\vert b\vert^2\Delta_{\del}+
a\bar b(\delbar\del^*+\del^*\delbar)+
b\bar a(\del\delbar^*+\delbar^*\del)\simeq
\vert a\vert^2\Delta_{\delbar}+\vert b\vert^2\Delta_{\del}
$$
hence $\Delta_{a,b}$ is elliptic.
\end{proof}

We denote with $\mathcal{H}^k_{D_{a,b}}(X):=
\Ker\,(\Delta_{{a,b}{\vert A^k}})$ the space of $D_{a,b}$-harmonic $k$-forms.
By the elliptic operators theory we get the following

\begin{theorem}\label{thm:Dab-harmonic}
Let $(X,J,g)$ be a compact almost-Hermitian manifold, then
the following Hodge decompositions holds, for every $k$,
$$
A^k(X)=\mathcal{H}^k_{D_{a,b}}(X)\oplus D_{a,b} A^{k-1}(X)\oplus
D_{a,b}^* A^{k+1}(X)\,.
$$
Moreover, the space $\mathcal{H}^\bullet_{D_{a,b}}(X)$ is finite-dimensional.
\end{theorem}

One has the following

\begin{theorem}\label{thm:Dab-harmonic-Dab-cohomology}
Let $(X,J,g)$ be a compact almost-Hermitian manifold, then there exists
an isomorphism, for every $k$,
$$
H^k_{D_{a,b}}(X)\simeq\mathcal{H}^k_{D_{a,b}}(X)\,.
$$
In particular, the space $H^\bullet_{D_{a,b}}(X)$ is finite-dimensional and we will denote with $h^\bullet_{D_{a,b}}(X)$ its dimension.
\end{theorem}

As a consequence we have the analogue of the Poincar\'e duality for the cohomology groups $H^\bullet_{D_{a,b}}(X)$.

\begin{prop}\label{prop:Dab-star}
Let $(X,J,g)$ be a compact almost-Hermitian manifold of dimension $2n$, then the Hodge-$*$-operator induces a duality isomoprhism, for every $k$,
$$
*:H^k_{D_{a,b}}(X)\to H^{2n-k}_{D_{a,b}}(X)\,.
$$
In particular, for every $k$, one has the equalities
$h^k_{D_{a,b}}(X)=h^{2n-k}_{D_{a,b}}(X)$.
\end{prop}

Similarly, one could develop a Hodge Theory for the Bott-Chern and Aeppli cohomologies of $(A^\bullet(X),D_{a,b},D_{c,e})$ (with $ae=bc$ and $(a,b)\neq (c,e)$)
following for instance \cite{schweitzer}.\\
In particular, the Bott-Chern and Aeppli Laplacians can be defined as
$$
\Delta_{BC_{a,b,c,e}}=(D_{a,b}D_{c,e})(D_{a,b}D_{c,e})^*+
(D_{a,b}D_{c,e})^*(D_{a,b}D_{c,e})+
(D_{c,e}^*D_{a,b})(D_{c,e}^*D_{a,b})^*+
$$
$$
(D_{c,e}^*D_{a,b})^*(D_{c,e}^*D_{a,b})+
D_{c,e}^*D_{c,e}+D_{a,b}^*D_{a,b}\,,
$$
$$\Delta_{A_{a,b,c,e}} \;:=\; 
D_{a,b}D_{a,b}^*+D_{c,e}D_{c,e}^*+
(D_{a,b}D_{c,e})^*(D_{a,b}D_{c,e})+
(D_{a,b}D_{c,e})(D_{a,b}D_{c,e})^*+
$$
$$
(D_{c,e}D_{a,b}^*)^*(D_{c,e}D_{a,b}^*)+
(D_{c,e}D_{a,b}^*)(D_{c,e}D_{a,b}^*)^*\,.
$$
These operators are elliptic and we denote with $\mathcal{H}^k_{BC(D_{a,b},D_{c,e})}(X):=\Ker(\Delta_{BC_{a,b,c,e}|A^k})$ the space of $BC$-harmonic $k$-forms and with $\mathcal{H}^k_{A(D_{a,b},D_{c,e})}(X):=\Ker(\Delta_{A_{a,b,c,e}\vert A^k})$ the space of $A$-harmonic $k$-forms.
By a direct calculation one can show the following
\begin{prop}
Let $(X,J,g)$ be a compact almost-Hermitian manifold. If 
 $ae=bc$ and $(a,b)\neq (c,e)$
then, a differential form $\alpha\in\mathcal{H}^k_{BC(D_{a,b},D_{c,e})}(X)$ if and only if
$$
D_{a,b}\,\alpha=0\,,\quad D_{c,e}\,\alpha=0\,,\quad
(D_{a,b}D_{c,e})^*\,\alpha=0\,.
$$
Similarly,
$\alpha\in\mathcal{H}^k_{A(D_{a,b},D_{c,e})}(X)$ if and only if
$$
(D_{a,b})^*\,\alpha=0\,,\quad (D_{c,e})^*\,\alpha=0\,,\quad
D_{a,b}D_{c,e}\,\alpha=0\,.
$$
\end{prop}

By the elliptic operators theory we get the following

\begin{theorem}\label{thm:BCA-harmonic}
Let $(X,J,g)$ be a compact almost-Hermitian manifold. If $ae=bc$ and $(a,b)\neq (c,e)$ then
the following Hodge decompositions hold, for every $k$,
$$
A^k(X)=\mathcal{H}^k_{BC(D_{a,b},D_{c,e})}(X)\oplus D_{a,b}D_{c,e} A^{k-2}(X)\oplus
(D_{c,e}^* A^{k+1}(X)+D_{a,b}^* A^{k+1}(X))\,,
$$
$$
A^k(X)=\mathcal{H}^k_{A(D_{a,b},D_{c,e})}(X)\oplus 
(D_{a,b} A^{k-1}(X)+D_{c,e} A^{k-1}(X))\oplus
((D_{a,b}D_{c,e})^* A^{k+2}(X))\,.
$$
Moreover, the spaces $\mathcal{H}^\bullet_{BC(D_{a,b},D_{c,e})}(X)$ and
$\mathcal{H}^\bullet_{A(D_{a,b},D_{c,e})}(X)$ are finite-dimensional.
\end{theorem}

One has the following

\begin{theorem}\label{thm:BCA-harmonic-BCA-cohomology}
Let $(X,J,g)$ be a compact almost-Hermitian manifold, then there exist
isomorphisms, for every $k$,
$$
H^k_{BC(D_{a,b},D_{c,e})}(X)\simeq\mathcal{H}^k_{BC(D_{a,b},D_{c,e})}(X)\,,
$$
and
$$
H^k_{A(D_{a,b},D_{c,e})}(X)\simeq\mathcal{H}^k_{A(D_{a,b},D_{c,e})}(X)\,.
$$
In particular, the spaces $H^\bullet_{BC(D_{a,b},D_{c,e})}(X)$ and
$H^\bullet_{A(D_{a,b},D_{c,e})}(X)$ are finite-dimensional.
\end{theorem}

\medskip

However, under some hypothesis on the parameters $a,b$ we can write down an explicit isomorphism.

\begin{prop}\label{prop:isom-parametric-derham}
Let $(X,J,g)$ be a compact almost-Hermitian manifold of dimension $2n$. Let
$a,b\in\mathbb{C}\setminus\left\lbrace 0\right\rbrace$ such that
$\vert a\vert=\vert b\vert$, then there exists an isomoprhism
$$
\text{Ker}\,\Delta_d\quad\simeq\quad \text{Ker}\,\Delta_{a,b}
$$
given by
$$
\alpha\mapsto\sum_{p+q=k}\left(\frac{a}{b}\right)^q\alpha^{p,q}
$$
where $\alpha^{p,q}$ denotes the $(p,q)$-component of a $k$-form $\alpha$.
\end{prop}

\begin{proof}
Let $\alpha=\sum_{p+q=k}\alpha^{p,q}$ be a $d$-closed $k$-form, namely
Hence $\mu\alpha+\del\alpha+\delbar\alpha+\bar\mu\alpha=0$.
Then, by bi-degree reasons
$$
\left\lbrace
\begin{array}{lcl}
\mu\,\alpha^{p+q,0} & =& 0\\
\del\,\alpha^{p+q,0} & =& -\mu\,\alpha^{p+q-1,1}\\
\del\,\alpha^{p+q-1,1} & =& -\delbar\,\alpha^{p+q,0} -\mu\,\alpha^{p+q-2,2}\\
\del\,\alpha^{p+q-2,2} & =& -\delbar\,\alpha^{p+q-1,1} 
-\mu\,\alpha^{p+q-3,3}-\bar\mu\,\alpha^{p+q,0}\\
\vdots & \vdots & \vdots \\
\del\,\alpha^{1,p+q-1} & =& -\delbar\,\alpha^{2,p+q-2} 
-\mu\,\alpha^{0,p+q}-\bar\mu\,\alpha^{3,p+q-3}\\
\del\,\alpha^{0,p+q} & =& -\delbar\,\alpha^{1,p+q-1} 
-\bar\mu\,\alpha^{2,p+q-2}\\
\delbar\,\alpha^{0,p+q} & =& -\bar\mu\,\alpha^{1,p+q-1}\\
\bar\mu\,\alpha^{0,p+q} & =& 0
\end{array}
\right.\,.
$$
Therefore,
$$
\left\lbrace
\begin{array}{lcl}
\mu\,\alpha^{p+q,0} & =& 0\\
b\del\,\alpha^{p+q,0} & =& -\frac{b^2}{a}\mu\,(\frac{a}{b}\alpha^{p+q-1,1})\\[5pt]
b\del\,(\frac{a}{b}\alpha^{p+q-1,1}) & =& -a\delbar\,\alpha^{p+q,0} -
\frac{b^2}{a}\mu\,(\frac{a^2}{b^2}\alpha^{p+q-2,2})\\[5pt]
b\del\,(\frac{a^2}{b^2}\alpha^{p+q-2,2}) & =& 
-a\delbar\,(\frac{a}{b}\alpha^{p+q-1,1}) 
-\frac{b^2}{a}\mu\,(\frac{a^3}{b^3}\alpha^{p+q-3,3})-
\frac{a^2}{b}\bar\mu\,\alpha^{p+q,0}\\
\vdots & \vdots & \vdots \\
b\del\,(\frac{a^{p+q-1}}{b^{p+q-1}}\alpha^{1,p+q-1}) & =& 
-a\delbar\,(\frac{a^{p+q-2}}{b^{p+q-2}}\alpha^{2,p+q-2} )
-\frac{b^2}{a}\mu\,(\frac{a^{p+q}}{b^{p+q}}\alpha^{0,p+q})
-\frac{a^2}{b}\bar\mu\,(\frac{a^{p+q-3}}{b^{p+q-3}}
\alpha^{3,p+q-3})\\[5pt]
b\del\,(\frac{a^{p+q}}{b^{p+q}}\alpha^{0,p+q}) & =& 
-a\delbar\,(\frac{a^{p+q-1}}{b^{p+q-1}}\alpha^{1,p+q-1}) 
-\frac{a^2}{b}\bar\mu\,(\frac{a^{p+q-2}}{b^{p+q-2}}
\alpha^{2,p+q-2})\\[5pt]
a\delbar\,(\frac{a^{p+q}}{b^{p+q}}\alpha^{0,p+q}) & =&
 -\frac{a^2}{b}\bar\mu\,(\frac{a^{p+q-1}}{b^{p+q-1}}
 \alpha^{1,p+q-1})\\[5pt]
\bar\mu\,\alpha^{0,p+q} & =& 0
\end{array}
\right.\,
$$
Namely, if $d\alpha=0$ then
$$
D_{a,b}\left(\alpha^{p+q,0}+\frac{a}{b}\alpha^{p+q-1,1}+
\frac{a^2}{b^2}\alpha^{p+q-2,2}+\cdots+
\frac{a^{p+q}}{b^{p+q}}\alpha^{0,p+q}\right)=0\,.
$$
Similarly, if $d^*\alpha=0$ then
$$
D_{a,b}^*\left(\alpha^{p+q,0}+\frac{\bar b}{\bar a}\alpha^{p+q-1,1}+
\frac{\bar b^2}{\bar a^2}\alpha^{p+q-2,2}+\cdots+
\frac{\bar b^{p+q}}{\bar a^{p+q}}\alpha^{0,p+q}\right)=0\,.
$$
Therefore if $\vert a\vert^2=\vert b\vert^2$ and $\Delta_d\,\alpha=0$ then 
$$
\alpha^{p+q,0}+\frac{a}{b}\alpha^{p+q-1,1}+
\frac{a^2}{b^2}\alpha^{p+q-2,2}+\cdots+
\frac{a^{p+q}}{b^{p+q}}\alpha^{0,p+q}
$$
is $\Delta_{a,b}$-harmonic.
\end{proof}

\begin{cor}
Let $(X,J,g)$ be a compact almost-Hermitian manifold of dimension $2n$. Let
$a,b\in\mathbb{C}\setminus\left\lbrace 0\right\rbrace$ such that
$\vert a\vert=\vert b\vert$, then there exists an isomoprhism
$$
H_{dR}^\bullet(X)\quad\simeq\quad H_{D_{a,b}}^\bullet(X).
$$
\end{cor}

Notice that in case of $D_{i,-i}=d^c$ the isomorphism becomes
$$
\alpha\mapsto\sum_{p+q=k}(-1)^q\alpha^{p,q}=i^{-k}J\alpha\,.
$$

\begin{rem}
If $D_a$ is a real operator, namely $D_a=D_{a,\bar a}$, then by previous corollary
there is an isomoprhism
$$
H_{dR}^\bullet(X)\quad\simeq\quad H_{D_{a}}^\bullet(X)
$$
for any $a\in\mathbb{C}\setminus\left\lbrace 0\right\rbrace$.
\end{rem}

\begin{ex}\label{example:kodaira-thurston-Dab-cohomology}
Let $\mathbb{H}(3;\mathbb{R})$ be the $3$-dimensional Heisenberg group and
$\mathbb{H}(3;\mathbb{Z})$ be the subgroup of matrices with entries in $\mathbb{Z}$.
The Kodaira-Thurston manifold is defined as the quotient
$$
X:=\left(\mathbb{H}(3;\mathbb{R})\times \mathbb{R}\right)/
\left(\mathbb{H}(3;\mathbb{Z})\times \mathbb{Z}\right)\,.
$$
The manifold $X$ is a $4$-dimensional nilmanifold which admits both complex and symplectic structures.
We consider the non-integrable almost-complex structure $J$ defined by the structure equations
$$
\left\lbrace
\begin{array}{lcl}
d\varphi^1 & =& 0\\
d\varphi^2 &=& \frac{1}{2i}\varphi^{12}+\frac{1}{2i}\left(\varphi^{1\bar 2}-\varphi^{2\bar 1}\right)+\frac{1}{2i}\varphi^{\bar 1\bar 2}
\end{array}
\right.\,
$$
where $\left\lbrace\varphi^1\,,\varphi^2\right\rbrace$ is a global co-frame of (1,0)-forms on $X$.\\
Hence, directly we get, for any $a,b\in\mathbb{C}\setminus\left\lbrace 0\right\rbrace$
$$
\left\lbrace
\begin{array}{lcl}
D_{a,b}\varphi^1 & =& 0\\
D_{a,b}\varphi^2 &=& \frac{1}{2i}a\left(\varphi^{1\bar2}-\varphi^{2\bar 1}\right)+\frac{1}{2i}b\varphi^{12}+
\frac{1}{2i}\frac{a^2}{b}\varphi^{\bar1\bar2}
\end{array}
\right.\,.
$$
We fix the $J$-Hermitian metric $\omega:=\frac{1}{2i}\sum_{j=1}^2\varphi^j\wedge\bar\varphi^j$, and by a direct computation one gets on invariant $2$-forms
$$
\text{Ker}\,D_{a,b,inv}=\mathbb{C}\left\langle
\varphi^{1\bar1},\varphi^{2\bar2},
\frac{b}{a}\varphi^{12}+\varphi^{1\bar2},
-\frac{b}{a}\varphi^{12}+\varphi^{2\bar1},
-\frac{b^2}{a^2}\varphi^{12}+\varphi^{\bar1\bar2}
\right\rangle
$$
and
$$
\text{Ker}\,D^*_{a,b,inv}=\mathbb{C}\left\langle
\varphi^{1\bar1},\varphi^{2\bar2},
-\frac{\bar a}{\bar b}\varphi^{12}+\varphi^{1\bar2},
\frac{\bar a}{\bar b}\varphi^{12}+\varphi^{2\bar1},
-\frac{\bar a^2}{\bar b^2}\varphi^{12}+\varphi^{\bar1\bar2}
\right\rangle\,.
$$
Therefore, one gets
$$
H^2_{D_{a,b},\text{inv}}\simeq\mathbb{C}\left\langle
\varphi^{1\bar1},\varphi^{2\bar2},
\varphi^{1\bar 2}+\varphi^{2\bar1},
\varphi^{\bar1\bar2}
-\frac{\vert a\vert^2-\vert b\vert^2}{a\bar b}\varphi^{1\bar2}
-\frac{b\bar a}{a\bar b}\varphi^{12}
\right\rangle\,,
$$
where we listed the harmonic representatives with respect to $\omega$.
In particular, for $a=b=1$ we get the harmonic representatives for the de Rham cohomology and for $a=-b=i$ we get the harmonic representatives for the
$d^c$-cohomology $H^2_{d^c}(X)$.
\end{ex}

\begin{rem}
Notice that if $J$ is integrable then $(A^\bullet(X),D_{a,b},D_{c,e})$ is a double complex for any choice of the parameters (provided $(a,b)\neq(c,e)$) and so one can define accordingly the associated Dolbeault, Bott-Chern and Aeppli cohomologies.
\end{rem}

\section{Differential operators on symplectic manifolds}\label{section:diff-op-sympl}

Let $(X,J,g,\omega)$ be a compact almost-K\"ahler manifold that is an almost-Hermitian manifold with fundamental form $\omega$ $d$-closed.
Then, we can generalize the symplectic cohomologies introduced in \cite{tseng-yau-I}.\\

Let
$$
L:=\omega\wedge-: A^{\bullet}(X)\to A^{\bullet+2}(X)
$$
and
$$
\Lambda:=-\star L\star: A^{\bullet}(X)\to A^{\bullet-2}(X)\,,
$$
where $\star=J*=*J$ is the symplectic-Hodge-$\star$-operator.
Denote with
$$
d^\Lambda:=[d,\Lambda]\,;
$$
since $\omega$ is symplectic we have that
$$
d^\Lambda=(-1)^{k+1}\star d\star_{\vert A^k(X)}
$$
i.e., $d^\Lambda$ is the Brylinski-codifferential (\cite{brylinski}), namely the symplectic adjoint of $d$.
Then, it is well known that $(d^c)^*=-d^\Lambda$, indeed on $k$-forms
$$
(d^c)^*=-*d^c*=-*J^{-1}dJ*=-(-1)^{k+1}*Jd\star=(-1)^k\star d\star=
-d^\Lambda\,.
$$

By the almost-K\"ahler identities (cf. Lemma \ref{lemma:almost-kahler-identities})

\begin{itemize}
\item[$\bullet$] $[\del,\Lambda]=i\,\delbar^*$ and $[\bar\mu,\Lambda]=i\,\mu^*$
\item[$\bullet$] $[\delbar,\Lambda]=-i\,\del^*$ and $[\mu,\Lambda]=-i\,\bar\mu^*$.
\end{itemize} 

one has the following

\begin{lemma}
Let $(X,J,g,\omega)$ be a compact almost-K\"ahler manifold, then for $a,b\in\mathbb{C}\setminus\left\lbrace 0\right\rbrace$,
\begin{itemize}
\item[$\bullet$] $[D_{a,b},L]=0$,
\item[$\bullet$] $[D_{a,b},\Lambda]=-i\,D^*_{-\bar b,\bar a}$.
\end{itemize}
Moreover,
$[D_{a,b},\Lambda]=(-1)^{k+1}\star D_{a,b}\star$ on $k$-forms if and only if
$D_{a,b}$ is a real operator.
\end{lemma}

\begin{proof}
By direct computations using the almost-K\"ahler identities
$$
[D_{a,b},\Lambda]=a[\delbar,\Lambda]+b[\del,\Lambda]+
\frac{b^2}{a}[\mu,\Lambda]+\frac{a^2}{b}[\bar\mu,\Lambda]=
$$
$$
=-ia\del^*+ib\delbar^*-i\frac{b^2}{a}\bar\mu^*+i\frac{a^2}{b}\mu^*=
-i\,D^*_{-\bar b,\bar a}\,.
$$
Moreover, notice that
$$
\star D_{a,b}\star=\bar a \star\delbar\star+\bar b\star\del\star+
\frac{\bar b^2}{\bar a}\star\mu\star+ \frac{\bar a^2}{\bar b}\star\bar\mu\star
$$
hence, $[D_{a,b},\Lambda]=(-1)^{k+1}\star D_{a,b}\star$ if and only if
$a=\bar b$ if and only if $D_{a,b}$ is a real operator by Lemma
\ref{lemma:real-operator}.
\end{proof}

As a consequence, we denote
$$
D_a^\Lambda:=[D_{a},\Lambda]=(-1)^{k+1}\star D_{a}\star_{\vert_{A^k(X)}}\,.
$$
This operator generalizes the Brylinski co-differential, indeed
$$
D_1^\Lambda=d^\Lambda\,.
$$
In fact using $D_a^\Lambda:=[D_{a},\Lambda]$ and $D_a^2=0$ we have that
$$
D_aD_a^\Lambda+D_a^\Lambda D_a=0\quad\text{and}\quad
(D_a^\Lambda)^2=0.
$$
In particular, for $a=1$ we recover the standard relations
$$
dd^\Lambda+d^\Lambda d=0\quad\text{and}\quad
(d^\Lambda)^2=0.
$$
Therefore, one can define
$$
H^\bullet_{D_a^\Lambda}:=
\frac{\text{Ker}\,D_a^\Lambda}{\text{Im}\,D_a^\Lambda},\quad
H_{BC(D_{a},D_a^\Lambda)}^\bullet(X):=
\frac{\text{Ker}\,D_{a}\cap\text{Ker}\,D_a^\Lambda}
{\text{Im}\,D_{a}D_a^\Lambda},\quad
H_{A(D_{a},D_a^\Lambda)}^\bullet(X):=
\frac{\text{Ker}\,D_{a}D_a^\Lambda}
{\text{Im}\,D_{a}+\text{Im}\,D_a^\Lambda}.
$$
The symplectic cohomologies defined in \cite{tseng-yau-I} correspond to the parameter $a=1$.

\section{Harmonic forms on almost-Hermitian manifolds}\label{section:harmonic-forms-almost-hermitian}

In the following we try to generalize the spaces of harmonic forms for the
Dolbeault, Bott-Chern and Aeppli cohomology groups of complex manifolds using the intrinsic decomposition of $d$ induced by the almost-complex structure. However, for a non-integrable almost-complex structure we do not have a cohomological counterpart (cf. also \cite{cirici-wilson-1}, \cite{cirici-wilson-2}).

Let $(X,\,J,\,g)$ be an almost-Hermitian manifold that means $X$ is a smooth manifold endowed with an almost complex structure $J$ and a
$J$-Hermitian metric $g$. As above denote with $*$ the associated Hodge-$*$-operator. 
Consequently,
$$
\delta^*=\del^*+\bar\mu^*\,,\qquad \bar\delta^*=\delbar^*+\mu^*
$$
and
$$
(d^c)^*=i(\delta^*-\bar\delta^*)=i(\del^*+\bar\mu^*-\delbar^*-\mu^*)\,.
$$

We define the following differential operators
$$
\Delta_{\bar\delta}:=\bar\delta\bar\delta^*+\bar\delta^*\bar\delta\,,
$$
$$
\Delta_{\delta}:=\delta\delta^*+\delta^*\delta\,,
$$
$$
\Delta_{BC(\delta,\bar\delta)}:=
(\delta\bar\delta)(\delta\bar\delta)^*+
(\delta\bar\delta)^*(\delta\bar\delta)+
(\bar\delta^*\delta)(\bar\delta^*\delta)^*+
(\bar\delta^*\delta)^*(\bar\delta^*\delta)+
\bar\delta^*\bar\delta+\delta^*\delta\,,
$$
$$\Delta_{A(\delta,\bar\delta)} \;:=\; \delta\delta^*+\deltabar\deltabar^*+(\delta\deltabar)^*(\delta\deltabar)+
(\delta\deltabar)(\delta\deltabar)^*+
(\deltabar\delta^*)^*(\deltabar\delta^*)+
(\deltabar\delta^*)(\deltabar\delta^*)^*\,.
$$

\begin{rem}
Notice that if $J$ is an integrable almost-complex structure
then these differential operators coincide with the classical Laplacian operators on complex manifolds, namely
the Dolbeault Laplacians
$$
\Delta_{\delbar}:=\delbar\delbar^*+\delbar^*\delbar\,,
$$
$$
\Delta_{\delta}:=\del\del^*+\del^*\del\,,
$$
and the Bott-Chern and Aeppli Laplacians
$$
\Delta_{BC}=(\del\delbar)(\del\delbar)^*+
(\del\delbar)^*(\del\delbar)+
(\delbar^*\del)(\delbar^*\del)^*+
(\delbar^*\del)^*(\delbar^*\del)+
\delbar^*\delbar+\del^*\del\,,
$$
$$\Delta_{A} \;:=\; \del\del^*+\delbar\delbar^*+(\del\delbar)^*(\del\delbar)+(\del\delbar)(\del\delbar)^*+(\delbar\del^*)^*(\delbar\del^*)+(\delbar\del^*)(\delbar\del^*)^*\,.
$$
\end{rem}

We have the following

\begin{prop}\label{prop:ellipticity-deltabar}
Let $(X,J,g)$ be an almost-Hermitian manifold, then the operators
$\Delta_{\bar\delta}$ and $\Delta_{\delta}$ are elliptic differential operators of the second order.
\end{prop}

\begin{proof}
The operator $\Delta_{\bar\delta}$ is elliptic, indeed it is a
lower order perturbation of its integrable counterpart.
More precisely, denoting with $\simeq$ the equivalence of the symbol of the operators we have
$$
\Delta_{\bar\delta}=
\deltabar^*\deltabar+\deltabar\deltabar^*\simeq
\delbar\delbar^*+\delbar^*\delbar
=\Delta_{\delbar}\,.
$$
Similar considerations can be done for $\Delta_{\delta}$.
\end{proof}

We denote with $\mathcal{H}^k_{\bar\delta}(X):=\Ker\Delta_{\bar\delta_{\vert A^k(X)}}$ the space of $\bar\delta$-harmonic $k$-forms and with
$\mathcal{H}^{p,q}_{\bar\delta}(X):=\Ker\Delta_{\bar\delta_{\vert A^{p,q}(X)}}$ the space of $\bar\delta$-harmonic $(p,q)$-forms, and similarly for the operator $\delta$.
We get the following

\begin{theorem}\label{thm:deltabar-harmonic}
Let $(X,J,g)$ be a compact almost-Hermitian manifold, then
the following Hodge decompositions hold
$$
A^k(X)=\mathcal{H}^k_{\bar\delta}(X)\oplus\bar\delta A^{k-1}(X)\oplus
\bar\delta^* A^{k+1}(X)
$$
and
$$
A^k(X)=\mathcal{H}^k_{\delta}(X)\oplus\delta A^{k-1}(X)\oplus
\delta^* A^{k+1}(X)
$$
Moreover, a $(p,q)$-form $\alpha\in\mathcal{H}^{p,q}_{\bar\delta}(X)$ if and only if
$\alpha\in\mathcal{H}^{p,q}_{\delbar}\cap\mathcal{H}^{p,q}_{\mu}$.
Similarly, a $(p,q)$-form $\alpha\in\mathcal{H}^{p,q}_{\delta}(X)$ if and only if
$\alpha\in\mathcal{H}^{p,q}_{\del}\cap\mathcal{H}^{p,q}_{\bar\mu}$.
\end{theorem}

\begin{proof}
The Hodge decompositions follow form the classical theory of elliptic operators.
Notice that a $k$-form $\beta$ is $\bar\delta$-harmonic
if and only if 
$$
\left\lbrace
\begin{array}{lcl}
\bar\delta\beta & =& 0\\
\bar\delta^*\beta &=& 0
\end{array}
\right.
\qquad\iff\qquad
\left\lbrace
\begin{array}{lcl}
\delbar\beta+\mu\beta & =& 0\\
\delbar^*\beta+\mu^*\beta &=& 0
\end{array}
\right.\,.
$$
Hence let $\alpha\in A^{p,q}(X)$, then $\alpha\in\Ker\Delta_{\bar\delta}$
if and only if 
$\delbar\alpha=0$, $\delbar^*\alpha=0$, $\mu\alpha=0$, $\mu^*\alpha=0$ concluding the proof.\\
\end{proof}

\begin{rem}
Since the operator $\Delta_{\bar\delta}$ is elliptic the associated space of harmonic forms $\mathcal{H}^{\bullet}_{\bar\delta}(X)$ is finite-dimensional on a compact almost-Hermitian manifold. In particular, we denote with $h^{\bullet}_{\bar\delta}(X)$ its dimension. The same applies for the operator $\delta$.
\end{rem}

\begin{prop}\label{prop:delta-delbar-delta-deltabar}
Let $(X,J,g)$ be a compact almost-Hermitian manifold, then
$$
\Delta_{\bar\delta}=\Delta_{\delbar}+\Delta_{\mu}+
[\delbar,\mu^*]+[\mu,\delbar^*]
$$
and
$$
\Delta_{\delta}=\Delta_{\del}+\Delta_{\bar\mu}+
[\del,\bar\mu^*]+[\bar\mu,\del^*]\,.
$$
In particular, $\mathcal{H}^{\bullet}_{\delbar}(X)
\cap\mathcal{H}^{\bullet}_{\mu}(X)\subseteq\mathcal{H}^{\bullet}_{\bar\delta}(X)$.
\end{prop}

\begin{proof}
We prove only the first equality since the second one can be easily obtained by conjugation.
We have
$$
\begin{aligned}
\Delta_{\bar\delta}=&\, (\delbar+\mu)(\delbar^*+\mu^*)+(\delbar^*+\mu^*)(\delbar+\mu)\\
=&\, \delbar\delbar^*+\delbar\mu^*+\mu\delbar^*+\mu\mu^*+
\delbar^*\delbar+\delbar^*\mu+\mu^*\delbar+\mu^*\mu\\
=&\, \Delta_{\delbar}+\Delta_{\mu}+
[\delbar,\mu^*]+[\mu,\delbar^*]\,.
\end{aligned}
$$
\end{proof}

\begin{rem}
Notice that in \cite{cirici-wilson-2} the authors consider on $2n$-dimensional compact almost-Hermitian manifolds the spaces
of harmonic forms
$\mathcal{H}^{\bullet,\bullet}_{\delbar}\cap\mathcal{H}^{\bullet,\bullet}_{\mu}$.
By Theorem \ref{thm:deltabar-harmonic} we know that
on bi-graded forms we are just reinterpreting these spaces since
$\mathcal{H}^{\bullet,\bullet}_{\delbar}(X)
\cap\mathcal{H}^{\bullet,\bullet}_{\mu}(X)=\mathcal{H}^{\bullet,\bullet}_{\bar\delta}(X)$. Hence we refer to
\cite{cirici-wilson-2} for the properties and several results concerning these spaces.
But in general,
we just proved that on total degrees we have only the inclusion $\mathcal{H}^{\bullet}_{\delbar}(X)
\cap\mathcal{H}^{\bullet}_{\mu}(X)\subseteq\mathcal{H}^{\bullet}_{\bar\delta}(X)$. In particular in Example \ref{example-1} we show that this inclusion can be strict.
\end{rem}

\begin{rem}\label{rem:delbar-hodge-duality}
Let $(X,J,g)$ be a compact almost-Hermitian manifold of real dimension $2n$, then the Hodge-$*$-operator induces duality isomorphisms for every $k$
$$
*:\mathcal{H}^k_{\bar\delta}(X)\to \mathcal{H}^{2n-k}_{\bar\delta}(X)\,,\quad
*:\mathcal{H}^k_{\delta}(X)\to \mathcal{H}^{2n-k}_{\delta}(X)\,.
$$
In particular, for every $p,q$
$$
*:\mathcal{H}^{p,q}_{\bar\delta}(X)\to \mathcal{H}^{n-p,n-q}_{\bar\delta}(X)\,,\quad
*:\mathcal{H}^{p,q}_{\delta}(X)\to \mathcal{H}^{n-p,n-q}_{\delta}(X)\,.
$$
This follows easily from the relations $*\Delta_{\bar\delta}=\Delta_{\bar\delta}*$
and $*\Delta_{\delta}=\Delta_{\delta}*$.\\
In particular, we have the usual symmetries for the Hodge diamonds, namely for every $k$
$$
h^k_{\bar\delta}(X)=h^{2n-k}_{\bar\delta}(X)\,,\quad
h^k_{\delta}(X)=h^{2n-k}_{\delta}(X)
$$
and for every $p,q$
$$
h^{p,q}_{\bar\delta}(X)=h^{n-p,n-q}_{\bar\delta}(X)\,,\quad
h^{p,q}_{\delta}(X)=h^{n-p,n-q}_{\delta}(X)\,.
$$
\end{rem}

\begin{prop}
Let $(X,J,g)$ be an almost-Hermitian manifold, then the operators
$\Delta_{BC(\delta,\bar\delta)}$ and $\Delta_{A(\delta,\bar\delta)}$ are elliptic differential operators of the fourth order.
\end{prop}
\begin{proof}
The calculations for the symbol of $\Delta_{BC(\delta,\bar\delta)}$ are similar to the ones for $\Delta_{\bar\delta}$ keeping only the highest order differential terms.
Denoting with $\simeq$ the equivalence of the symbol of the operators we have
$$
\Delta_{BC(\delta,\bar\delta)}\simeq
\delta\bar\delta\bar\delta^*\delta^*+
\bar\delta^*\delta^*\delta\bar\delta+
\bar\delta^*\delta\delta^*\bar\delta+
\delta^*\bar\delta\bar\delta^*\delta
\simeq\delta\delta^*\deltabar\deltabar^*+
\delta^*\delta\deltabar^*\deltabar+
\delta\delta^*\deltabar^*\deltabar+
\delta^*\delta\deltabar\deltabar^*
$$
$$
\simeq\left(\delta^*\delta+\delta\delta^*\right)
\left(\deltabar^*\deltabar+\deltabar\deltabar^*\right)=
\Delta_{\delta}\Delta_{\deltabar}\simeq\Delta_{\deltabar}^2\,.
$$
Similar considerations can be done for $\Delta_{A(\delta,\bar\delta)}$.
\end{proof}

We denote with $\mathcal{H}^k_{BC(\delta,\bar\delta)}(X):=\Ker(\Delta_{BC(\delta,\bar\delta)\vert A^k})$ the space of $\Delta_{BC(\delta,\bar\delta)}$-harmonic
$k$-forms and with $\mathcal{H}^{p,q}_{BC(\delta,\bar\delta)}(X):=\Ker(\Delta_{BC(\delta,\bar\delta){\vert A^{p,q}}})$
the space of $\Delta_{BC(\delta,\bar\delta)}$-harmonic
$(p,q)$-forms. We have the following Lemma whose proof is a direct computation.
\begin{lemma}\label{lemma:equivalence-bc-harmonic}
Let $(X,J,g)$ be a compact almost-Hermitian manifold. Then,
a differential form $\alpha\in\mathcal{H}^k_{BC(\delta,\bar\delta)}(X)$ if and only if
$$
\delta\alpha=0\,,\quad \bar\delta\alpha=0\,,\quad
(\delta\bar\delta)^*\alpha=0\,.
$$
\end{lemma}

 We get the following

\begin{prop}
Let $(X,J,g)$ be a compact almost-Hermitian manifold, then the following Hodge decomposition holds
$$
A^k(X)=\mathcal{H}^k_{BC(\delta,\bar\delta)}(X)
\stackrel{\perp}{\oplus}
\left(\delta\bar\delta A^{k-2}(X)\oplus
\left(
\bar\delta^* A^{k+1}(X)+\delta^* A^{k+1}(X)\right)\right)\,.
$$
Moreover, a $(p,q)$-form $\alpha\in\mathcal{H}^{p,q}_{BC(\delta,\bar\delta)}(X)$ if and only if
$$
\left\lbrace
\begin{array}{lcl}
\del\alpha & =& 0\\
\delbar\alpha &=& 0\\
\mu\alpha &=& 0\\
\bar\mu\alpha &=& 0\\
(\del\delbar+\bar\mu\mu)(*\alpha) &=& 0\\
\del\mu(*\alpha) &=& 0\\
\bar\mu\delbar(*\alpha) &=& 0
\end{array}
\right.\,.
$$
\end{prop}

\begin{proof}
The Hodge decomposition follows from the ellipticity of $\Delta_{BC(\delta,\bar\delta)}$.\\
Now let $\alpha\in A^k(X)$, then in view of Lemma \ref{lemma:equivalence-bc-harmonic}
$\alpha\in\Ker\Delta_{BC(\delta,\bar\delta)}$
if and only if
$$
\left\lbrace
\begin{array}{lcl}
\delta\alpha & =& 0\\
\bar\delta\alpha &=& 0\\
\delta\bar\delta*\alpha &=& 0
\end{array}
\right.\,.
$$
if and only if 
$$
\left\lbrace
\begin{array}{lcl}
(\del+\bar\mu)\alpha & =& 0\\
(\delbar+\mu)\alpha &=& 0\\
(\del\delbar+\del\mu+\bar\mu\delbar+\bar\mu\mu)(*\alpha)&=& 0\\
\end{array}
\right.\,.
$$
In particular, if $\alpha$ is a $(p,q)$-form we obtain the thesis.\\
Finally, given $\alpha\in \mathcal{H}^k_{BC(\delta,\bar\delta)}(X)$, $\beta\in A^{k-2}(X)$, $\gamma\in A^{k+1}(X)$ and $\eta\in A^{k+1}(X)$
we have
$$
(\alpha,\delta\bar\delta\beta+\deltabar^*\gamma+\delta^*\eta)=
((\delta\bar\delta)^*\alpha,\beta)+(\deltabar\alpha,\gamma)
+(\delta\alpha,\eta)=0\,.
$$
 \end{proof}
 
\begin{rem}
Notice that the spaces
$\delta\bar\delta A^{k-2}(X)$ and
$\bar\delta^* A^{k+1}(X)+\delta^* A^{k+1}(X)$ are orthogonal if and only if $\delta^2=0$.
\end{rem}
 
Similarly, if we denote with $\mathcal{H}^k_{A(\delta,\bar\delta)}(X):=\Ker(\Delta_{A(\delta,\bar\delta){\vert A^k}})$ the space of 
$\Delta_{A(\delta,\bar\delta)}$-harmonic
$k$-forms and with $\mathcal{H}^{p,q}_{A(\delta,\bar\delta)}(X):=\Ker(\Delta_{A(\delta,\bar\delta){\vert A^{p,q}}})$
the space of $\Delta_{A(\delta,\bar\delta)}$-harmonic
$(p,q)$-forms we get the following

\begin{prop}
Let $(X,J,g)$ be a compact almost-Hermitian manifold, then the following Hodge decomposition holds
$$
A^k(X)=\mathcal{H}^k_{A(\delta,\bar\delta)}(X)\stackrel{\perp}{\oplus}
\left(
\left(\delta A^{k-1}(X)+\deltabar A^{k-1}(X)\right)
\oplus \left(\delta\deltabar\right)^* A^{k+2}(X)\right)\,.
$$
Moreover, a $(p,q)$-form $\alpha\in\mathcal{H}^{p,q}_{A(\delta,\bar\delta)}(X)$ if and only if
$$
\left\lbrace
\begin{array}{lcl}
\del^*\alpha & =& 0\\
\delbar^*\alpha &=& 0\\
\mu^*\alpha &=& 0\\
\bar\mu^*\alpha &=& 0\\
(\del\delbar+\bar\mu\mu)\alpha &=& 0\\
\del\mu\alpha &=& 0\\
\bar\mu\delbar\alpha &=& 0
\end{array}
\right.\,.
$$
\end{prop}

 \begin{rem}
Since the operators $\Delta_{BC(\delta,\bar\delta)}$ and $\Delta_{A(\delta,\bar\delta)}$ are elliptic, the associated spaces of harmonic forms $\mathcal{H}^{\bullet}_{BC(\delta,\bar\delta)}(X)$, $\mathcal{H}^{\bullet}_{A(\delta,\bar\delta)}(X)$ are finite-dimensional on a compact almost-Hermitian manifold. In particular, we denote with $h^{\bullet}_{BC(\delta,\bar\delta)}(X)$ and $h^{\bullet}_{A(\delta,\bar\delta)}(X)$ their dimensions.
\end{rem}

\begin{rem}
Let $(X,J,g)$ be an almost-Hermitian manifold. Then, by definition,
conjugation induces the following isomorphisms
$$
\overline{\mathcal{H}^{\bullet}_{\bar\delta}(X)}=\mathcal{H}^\bullet_{\delta}(X)\,,
\qquad
\overline{\mathcal{H}^{\bullet}_{BC(\delta,\bar\delta)}(X)}=\mathcal{H}^\bullet_{BC(\delta,\bar\delta)}(X)\,.
$$
In particular, for any $p,\,q$
$$
\overline{\mathcal{H}^{p,q}_{\bar\delta}(X)}=\mathcal{H}^{q,p}_{\delta}(X)\,,
\qquad
\overline{\mathcal{H}^{p,q}_{BC(\delta,\bar\delta)}(X)}=\mathcal{H}^{q,p}_{BC(\delta,\bar\delta)}(X)\,.
$$
Therefore, we have the following dimensional equalities
for every $k$
$$
h^k_{\bar\delta}(X)=h^{k}_{\delta}(X)
$$
and for every $p,q$
$$
h^{p,q}_{\bar\delta}(X)=h^{q,p}_{\delta}(X)\,,\quad
h^{p,q}_{BC(\delta,\bar\delta)}(X)=h^{q,p}_{BC(\delta,\bar\delta)}(X)\,.
$$
\end{rem}

\begin{rem}\label{rem:BC-hodge-duality}
Let $(X,J,g)$ be a compact almost-Hermitian manifold of real dimension $2n$, then the Hodge-$*$-operator induces duality isomorphisms for every $k$
$$
*:\mathcal{H}^k_{BC(\delta,\bar\delta)}(X)\to \mathcal{H}^{2n-k}_{A(\delta,\bar\delta)}(X)\,.
$$
In particular, for every $p,q$
$$
*:\mathcal{H}^{p,q}_{BC(\delta,\bar\delta)}(X)\to \mathcal{H}^{n-p,n-q}_{A(\delta,\bar\delta)}(X)\,.
$$
Therefore we have the usual symmetries for the Hodge diamonds, namely for every $k$
$$
h^k_{BC(\delta,\bar\delta)}(X)=h^{2n-k}_{A(\delta,\bar\delta)}(X)
$$
and for every $p,q$
$$
h^{p,q}_{BC(\delta,\bar\delta)}(X)=h^{n-p,n-q}_{A(\delta,\bar\delta)}(X)\,.
$$
\end{rem}

\section{Harmonic forms on almost-K\"ahler manifolds}\label{section:harmonic-forms-almost-kahler}

Let $(X,J,g,\omega)$ be a compact almost-K\"ahler manifold.
With the usual notations, we have the following almost-K\"ahler identities
(cf. \cite{debartolomeis-tomassini}, \cite{cirici-wilson-2})

\begin{lemma}\label{lemma:almost-kahler-identities}
Let $(X,J,g,\omega)$ be an almost-K\"ahler manifold then
\begin{itemize}
\item[$\bullet$] $[\delta,\Lambda]=i\,\bar\delta^*$, 
$[\del,\Lambda]=i\,\delbar^*$ and $[\bar\mu,\Lambda]=i\,\mu^*$
\item[$\bullet$] $[\bar\delta,\Lambda]=-i\,\delta^*$, 
$[\delbar,\Lambda]=-i\,\del^*$ and $[\mu,\Lambda]=-i\,\bar\mu^*$.
\end{itemize} 
\end{lemma}

\begin{proof}
For the sake of completeness we recall here the proof.
We have
$$
d^\Lambda=[d,\Lambda]=[\delta+\deltabar,\Lambda]=
[\del+\bar\mu+\delbar+\mu,\Lambda]
$$
and
$$
-(d^c)^*=i(\deltabar^*-\delta^*)=i(\delbar^*+\mu^*-\del^*-\bar\mu^*).
$$
Since $\omega$ is symplectic, $d^\Lambda=-(d^c)^*$ as recalled at the beginning of Section \ref{section:diff-op-sympl};
hence $[\delta,\Lambda]=i\,\bar\delta^*$ and $[\bar\delta,\Lambda]=-i\,\delta^*$.
\end{proof}

As a consequence one has the following (see \cite[Lemma 3.6]{debartolomeis-tomassini})

\begin{prop}\label{prop:deltabar-delta-laplacians}
Let $(X,J,g,\omega)$ be an almost-K\"ahler manifold, then
$\Delta_{\bar\delta}$ and $\Delta_{\delta}$ are related by 
$$
\Delta_{\bar\delta}=\Delta_{\delta}
$$
and
$$
\Delta_d=\Delta_{\bar\delta}+\Delta_{\delta}+E_J
$$
where
$$
E_J=\delta\bar\delta^*+\bar\delta^*\delta+\bar\delta\delta^*+\delta^*\bar\delta\,.\\
$$
In particular, their spaces of harmonic forms coincide, i.e. $\mathcal{H}^{\bullet}_{\delta}(X)=\mathcal{H}^{\bullet}_{\bar\delta}(X)$\,.
\end{prop}

In fact, we can use this result to characterize K\"ahler manifolds among the almost-K\"ahler ones.
\begin{cor}\label{cor:kahler-equality-laplacians}
Let $(X,J,g,\omega)$ be a compact almost-K\"ahler manifold, then
$$
\Delta_d=2\Delta_{\delta}\quad\iff\quad (X,J,g,\omega) \text{ is K\"ahler.}
$$
\end{cor}
\begin{proof}
First of all, on any almost-K\"ahler manifold one has (cf. e.g., \cite{cirici-wilson-2})
$$
[\Delta_d,L]=[[d,d^*],L]=[d,[d^*,L]]=-[d,d^c].
$$
In view of Lemma \ref{lemma:almost-kahler-identities}, it is 
$$[\delta,\Lambda]=i\,\bar\delta^*;
$$
therefore, taking the adjoint,
$$
[L,\delta^*]=-i\deltabar.
$$
Furthermore, since $\omega$ is $d$-closed, we have 
$$[\delta,L]=0.
$$
Hence, (cf. Lemma \ref{lemma:commutator-laplacian-L})
$$
[\Delta_\delta,L]=[[\delta,\delta^*],L]=[\delta,[\delta^*,L]]=i[\delta,\deltabar]=0,
$$
that is 
$$
 [\Delta_\delta,L]=0.
$$
By Proposition \ref{prop:deltabar-delta-laplacians}, on an almost-K\"ahler manifold we have that 
$$
\Delta_d=\Delta_{\bar\delta}+\Delta_{\delta}+E_J=2\Delta_{\delta}+E_J,
$$
and we want to show that $E_J=0$ if and only if $J$ is integrable.\newline
Clearly, if $J$ is integrable, then $(X,J,g,\omega)$ is K\"ahler and as a consequence of the K\"ahler identities, 
$\Delta_d=2\Delta_{\del}$. \newline 
For the converse implication,  assume that $\Delta_d=2\Delta_{\delta}$. Then, by the above formula,
$$
dd^c+d^cd=-[\Delta_d,L]=-2[\Delta_\delta,L]=0,
$$
and, as noticed in Section \ref{section:preliminaries}, $d$ and $d^c$ anticommute if and only if $J$ is integrable.
\end{proof}

An immediate consequence of Proposition \ref{prop:deltabar-delta-laplacians} is also the following

\begin{cor}\label{cor:comparison-betti-numbers}
Let $(X,J,g,\omega)$ be a compact almost-K\"ahler manifold, then
$$
\mathcal{H}^\bullet_{\bar\delta}(X)\subseteq \mathcal{H}^\bullet_{dR}(X)\,,
$$
namely every $\bar\delta$-harmonic form is harmonic. In particular,
$$
h^\bullet_{\bar\delta}(X)\leq b_\bullet(X)\,,
$$
where $b_\bullet(X)$ denotes the Betti numbers of $X$.
\end{cor}

We will see with an explicit example that the inequality $
h^\bullet_{\bar\delta}(X)\leq b_\bullet(X)
$ does not hold for an arbitrary compact almost-Hermitian manifold.

\begin{lemma}\label{lemma:sympl-harmonic-deltabar-harmonic}
Let $(X,J,g,\omega)$ be an almost-K\"ahler manifold, then $d^\Lambda=i(\bar\delta^*-\delta^*)$. In particular, a $(p,q)$-form is symplectic harmonic, i.e., it belongs to $\Ker d\cap\Ker d^\Lambda$, if and only if belongs to $\Ker\delta\cap\Ker\bar\delta\cap\Ker\delta^*\cap\Ker\bar\delta^*$.
\end{lemma}

\begin{proof}
Since $d^c=-i(-\delbar+\del+\bar\mu-\mu)=i(\bar\delta-\delta)$
we have that
$(d^c)^*=i(-\delbar^*+\del^*+\bar\mu^*-\mu^*)=
i(\delta^*-\bar\delta^*)$ and so the thesis follows from $d^\Lambda=-(d^c)^*$ as noted in Section \ref{section:diff-op-sympl}.
\end{proof}

In general, the existence of a symplectic harmonic representative in every de-Rham cohomology class is equivalent to the Hard-Lefschetz condition (cf. \cite{brylinski},
\cite{yan},
\cite{mathieu}, \cite{tseng-yau-I}). Therefore, Tseng and Yau in \cite{tseng-yau-I} introduced the space
\[
H^k_{d+d^\Lambda}\left(X\right)
:=\frac{\ker(d+d^\Lambda)\cap A^k(X)}{\Imm dd^\Lambda\cap A^k(X)},
\]
and they study Hodge theory for it.
It turns out that $H^k_{d+d^\Lambda}\left(X\right)\simeq
\mathcal{H}^k_{d+d^\Lambda}\left(X\right)$ where
$$
\mathcal{H}^k_{d+d^\Lambda}\left(X\right)=
\text{Ker}\,d\cap \text{Ker}\,d^\Lambda\cap
\text{Ker}\,(dd^\Lambda)^*\,.
$$
Let us denote with $\mathcal{H}^{p,q}_{d+d^\Lambda}\left(X\right)$
the $(d+d^\Lambda)$-harmonic $(p,q)$-forms.

\begin{rem}
Notice that on a compact almost-K\"ahler manifold $(X^{2n},J,g,\omega)$ we have the inclusion
$$
\mathcal{H}^\bullet_{\deltabar}(X)\subseteq\mathcal{H}^\bullet_{d+d^\Lambda}(X),
$$
indeed if $\alpha\in\mathcal{H}^\bullet_{\deltabar}(X)$ then, by Proposition \ref{prop:deltabar-delta-laplacians},
$\alpha\in\mathcal{H}^\bullet_{\delta}(X)$, namely
$\delta\alpha=0$, $\delta^*\alpha=0$, $\deltabar\alpha=0$ and $\deltabar^*\alpha=0$.
Since $d=\delta+\deltabar$ and
$d^\Lambda=-(d^c)^*=-i(\delta^*-\deltabar^*)$ then we have the inclusion.\\
Moreover, if $J$ is $\mathcal{C}^{\infty}$-pure and full \cite{LZ} (e.g., this is always the case if $n=2$, see \cite{DLZ}) by Corollary \ref{cor:comparison-betti-numbers} and \cite[Theorem 4.2]{tardini-tomassini-instability}
one has
$$
\mathcal{H}^2_{\bar\delta}(X)\subseteq \mathcal{H}^2_{dR}(X)
\subseteq\mathcal{H}^2_{d+d^\Lambda}(X)\,
$$
and in particular, $h^2_{\bar\delta}(X)\leq b_2(X)\leq
 h^2_{d+d^\Lambda}(X).$
Recall that if $n=2$ by \cite[Theorem 4.5]{tardini-tomassini-proper-surjective} (cf. also \cite[Section 3.2]{tardini-proceeding})  $b_2(X)<h^2_{d+d^\Lambda}(X)$ unless $(X,\omega)$, as a symplectic manifold, satisfies the Hard Lefschetz condition.
\end{rem}

On bigraded forms we have a different situation from Corollary
\ref{cor:comparison-betti-numbers}.

\begin{theorem}\label{thm:equalities-harmonic-spaces}
Let $(X,J,g,\omega)$ be a compact almost-K\"ahler manifold, then on $(p,q)$-forms
$$
\mathcal{H}^{p,q}_{d+d^\Lambda}(X)=
\mathcal{H}^{p,q}_{\delta}(X)\cap\mathcal{H}^{p,q}_{\bar\delta}(X)=
\mathcal{H}^{p,q}_{\delbar}(X)\cap \mathcal{H}^{p,q}_{\del}(X)\cap
\mathcal{H}^{p,q}_{\bar\mu}(X)\cap\mathcal{H}^{p,q}_{\mu}(X)=
$$
$$
=\mathcal{H}^{p,q}_{\delbar}(X)\cap\mathcal{H}^{p,q}_{\mu}(X)=
\mathcal{H}^{p,q}_{d}(X)\,.
$$
\end{theorem}

\begin{proof}
Notice that the equality $\mathcal{H}^{p,q}_{\delta}(X)\cap\mathcal{H}^{p,q}_{\bar\delta}(X)=
\mathcal{H}^{p,q}_{\delbar}(X)\cap \mathcal{H}^{p,q}_{\del}(X)\cap
\mathcal{H}^{p,q}_{\bar\mu}(X)\cap\mathcal{H}^{p,q}_{\mu}(X)$
follows from Theorem \ref{thm:deltabar-harmonic}.\\
The equalities $\mathcal{H}^{p,q}_{\delbar}(X)\cap \mathcal{H}^{p,q}_{\del}(X)\cap
\mathcal{H}^{p,q}_{\bar\mu}(X)\cap\mathcal{H}^{p,q}_{\mu}(X)
=\mathcal{H}^{p,q}_{\delbar}(X)\cap\mathcal{H}^{p,q}_{\mu}(X)=
\mathcal{H}^{p,q}_{d}(X)
$
follow from \cite[Proposition 3.3, Theorem 4.3]{cirici-wilson-2}.\\
Indeed,
$$
\mathcal{H}^{p,q}_{\delbar}(X)\cap \mathcal{H}^{p,q}_{\del}(X)\cap
\mathcal{H}^{p,q}_{\bar\mu}(X)\cap\mathcal{H}^{p,q}_{\mu}(X)=
\text{Ker}\,(\Delta_{\delbar}+\Delta_{\del}+\Delta_{\bar\mu}+
\Delta_{\mu})
$$
$$
=\text{Ker}\,(\Delta_{\delbar}+\Delta_{\mu})\cap
\text{Ker}\,(\Delta_{\del}+\Delta_{\bar\mu})=
\text{Ker}\,(\Delta_{\delbar}+\Delta_{\mu})=
\text{Ker}\,(\Delta_{\delbar})\cap \text{Ker}\,(\Delta_{\mu})\,.
$$
\\
We just need to prove that $\mathcal{H}^{p,q}_{d+d^\Lambda}(X)=
\mathcal{H}^{p,q}_{\delta}(X)\cap\mathcal{H}^{p,q}_{\bar\delta}(X)$.
Let $\alpha\in\mathcal{H}^{p,q}_{d+d^\Lambda}(X)$, then $d\alpha=0$,
$d^\Lambda\alpha=0$ and $dd^\Lambda*\alpha=0$, or equivalently
$d\alpha=0$, $d^c*\alpha=0$ and $d*d^c\alpha=0$. Since on $(p,q)$-forms
$d\alpha=0$ implies $d^c\alpha=0$ the last condition is superfluous, and $d\alpha=0$, $d^c*\alpha=0$ is equivalent to $\delta\alpha=0$, $\bar\delta\alpha=0$, $\delta*\alpha=0$, $\bar\delta*\alpha=0$ (cf. Lemma
\ref{lemma:sympl-harmonic-deltabar-harmonic}).
\end{proof}

\begin{theorem}
Let $(X,J,g,\omega)$ be an almost-K\"ahler manifold, then
$\Delta_{BC(\delta,\bar\delta)}$, $\Delta_{\bar\delta}$ are related by 
$$
\Delta_{BC(\delta,\bar\delta)}=\Delta_{\bar\delta}^2+\bar\delta^*\bar\delta+\delta^*\delta+F_J
$$
where
$$
F_J:=-\delta\left(\delta\bar\delta^*+\bar\delta^*\delta\right)\bar\delta^*+
\left(\delta\bar\delta^*+\bar\delta^*\delta\right)\delta^*\bar\delta+
\delta^*\bar\delta\left(\delta\bar\delta^*+\bar\delta^*\delta\right)-\delta^*\left(\delta\bar\delta^*+\bar\delta^*\delta\right)\bar\delta\,.
$$
\end{theorem}

\begin{proof}
First of all, since it will be useful in the following, we notice that by the almost-K\"ahler identities 
$\delta^*=i\,[\deltabar,\Lambda]$ and $\deltabar^*=-i\,[\delta,\Lambda]$
we obtain
$$
\delta^*\deltabar=i(\deltabar\Lambda\deltabar-\Lambda\deltabar^2)\,,\quad
\deltabar\delta^*=i(\deltabar^2\Lambda-\deltabar\Lambda\deltabar)
$$
and similarly for their conjugates.\\
Recall that when $J$ is non-integrable $\delta^2\neq 0$ and $\bar\delta^2\neq 0$ and so we cannot cancel them out in these expressions. 
By Proposition \ref{prop:deltabar-delta-laplacians}
$$
\Delta_{\bar\delta}^2=\Delta_{\delta}\Delta_{\bar\delta}=
\delta\delta^*\deltabar\deltabar^*+
\delta\delta^*\deltabar^*\deltabar+
\delta^*\delta\deltabar\deltabar^*+
\delta^*\delta\deltabar^*\deltabar.
$$
Now in the first and fourth terms we use the previous formulas, and in the second and third terms we use the fact that $\delta$ and $\deltabar$ anticommute. Hence, we get
$$
\Delta_{\bar\delta}^2=
\delta(i\deltabar\Lambda\deltabar-i\Lambda\deltabar^2)\deltabar^*
-\delta\deltabar^*\delta^*\deltabar
-\delta^*\deltabar\delta\deltabar^*
+\delta^*(-i\delta^2\Lambda+i\delta\Lambda\delta)\deltabar\,.
$$
Using again that $\delta^*=i\,[\deltabar,\Lambda]$ one has
$$
i\delta\deltabar\Lambda\deltabar\deltabar^*=
\delta\deltabar(-\delta^*\deltabar^*+i\deltabar\Lambda\deltabar^*)=
i\delta\deltabar^2\Lambda\deltabar^*+\delta\deltabar\deltabar^*\delta^*
$$
and so the first term in the previous expression of $\Delta_{\bar\delta}^2$
becomes
$$
\delta(i\deltabar\Lambda\deltabar-i\Lambda\deltabar^2)\deltabar^*=
i\delta\deltabar^2\Lambda\deltabar^*+
\delta\deltabar\deltabar^*\delta^*
-i\delta\Lambda\deltabar^2\deltabar^*=
\delta\deltabar\deltabar^*\delta^*+
\delta(i\deltabar^2\Lambda-i\Lambda\deltabar^2)\deltabar^*
$$
$$
=\delta\deltabar\deltabar^*\delta^*+
\delta(\delta^*\deltabar+\deltabar\delta^*)\deltabar^*
$$
and similarly the fourth term becomes
$$
\delta^*(-i\delta^2\Lambda+i\delta\Lambda\delta)\deltabar=
-i\delta^*\delta^2\Lambda\deltabar+
\deltabar^*\delta^*\delta\deltabar
+i\delta^*\Lambda\delta^2\deltabar=
\deltabar^*\delta^*\delta\deltabar+
\delta^*(-i\delta^2\Lambda+i\Lambda\delta^2)\deltabar
$$
$$
=\deltabar^*\delta^*\delta\deltabar+
\delta^*(\delta\deltabar^*+\deltabar^*\delta)\deltabar\,.
$$
For the second term using again from the K\"ahler identities that
$\delta\deltabar^*=-i\delta^2\Lambda+i\delta\Lambda\delta$ and
$\deltabar^*=-i\,[\delta,\Lambda]$ one has
$$
-\delta\deltabar^*\delta^*\deltabar=
i\delta^2\Lambda\delta^*\deltabar-i\delta\Lambda\delta\delta^*\deltabar=
i\delta^2\Lambda\delta^*\deltabar+\deltabar^*\delta\delta^*\deltabar
-i\Lambda\delta^2\delta^*\deltabar\,=
\deltabar^*\delta\delta^*\deltabar+
(i\delta^2\Lambda-i\Lambda\delta^2)\delta^*\deltabar
$$
$$
=\deltabar^*\delta\delta^*\deltabar-
(\delta\deltabar^*+\deltabar^*\delta)\delta^*\deltabar
$$
and similarly for the third term
$$
-\delta^*\deltabar\delta\deltabar^*=
i\delta^*\deltabar\delta^2\Lambda+
\delta^*\deltabar\deltabar^*\delta
-i\delta^*\deltabar\Lambda\delta^2=
\delta^*\deltabar\deltabar^*\delta+
\delta^*\deltabar(i\delta^2\Lambda-i\Lambda\delta^2)
$$
$$
=\delta^*\deltabar\deltabar^*\delta-
\delta^*\deltabar(\delta\deltabar^*+\deltabar^*\delta)
\,.
$$
Putting all this together we obtain
$$
\Delta_{\bar\delta}^2=
\Delta_{BC(\delta,\deltabar)}-\deltabar^*\deltabar-\delta^*\delta-F_J
$$
concluding the proof.
Here in the expression of $F_J$ we have used that
$$
\delta^*\bar\delta+\bar\delta\delta^*=
\delta\bar\delta^*+\bar\delta^*\delta
$$
We prove this last statement separately in the following Proposition.
\end{proof}

Clearly, if $J$ is integrable we recover the classical relations between the Bott-Chern and Dolbeault Laplacians (cf. e.g., \cite{schweitzer}), namely on K\"ahler manifolds
$$
\Delta_{BC}=\Delta_{\delbar}^2+\delbar^*\delbar+\del^*\del\,.
$$
In particular, $F_J=0$ since by the K\"ahler identities $\del\delbar^*+\delbar^*\del=0$.

\begin{prop}
Let $(X,J,g,\omega)$ be an almost-K\"ahler manifold, then
$$
\bar\delta\delta^*+\delta^*\bar\delta=
\delta\bar\delta^*+\bar\delta^*\delta\,.
$$
In particular,
\begin{itemize}
\item[$\bullet$]
$
\delta\bar\delta^*+\bar\delta^*\delta=
\del\delbar^*+\delbar^*\del+\delbar\del^*+\del^*\delbar\,,
$
\item[$\bullet$] $E_J=2(\delta\bar\delta^*+\bar\delta^*\delta)$.
\end{itemize}
\end{prop}

\begin{proof}
We have
$$
\bar\delta\delta^*+\delta^*\bar\delta=
(\delbar+\mu)(\del^*+\bar\mu^*)+(\del^*+\bar\mu^*)(\delbar+\mu)=
$$
$$
=\delbar\del^*+\del^*\delbar+\delbar\bar\mu^*+\bar\mu^*\delbar+\mu\del^*+\del^*\mu+\mu\bar\mu^*+\bar\mu^*\mu\,.
$$
Now, by \cite[Lemma 3.7]{debartolomeis-tomassini} we have
$$
\mu\bar\mu^*+\bar\mu^*\mu=0
$$
and
$$
\delbar\del^*+\del^*\delbar=\del\mu^*+\mu^*\del+\bar\mu\delbar^*+\delbar^*\bar\mu\,,
$$
hence
$$
\bar\delta\delta^*+\delta^*\bar\delta=
\del\mu^*+\mu^*\del+\bar\mu\delbar^*+\delbar^*\bar\mu+
\delbar\bar\mu^*+\bar\mu^*\delbar+\mu\del^*+\del^*\mu\,.
$$
Using conjugation we have
$$
\delta\bar\delta^*+\bar\delta^*\delta=
\del\delbar^*+\delbar^*\del+\del\mu^*+\mu^*\del+\bar\mu\delbar^*+\delbar^*\bar\mu=
$$
$$
=\mu\del^*+\del^*\mu+\delbar\bar\mu^*+\bar\mu^*\delbar+
\del\mu^*+\mu^*\del+\bar\mu\delbar^*+\delbar^*\bar\mu\,,
$$
therefore $\bar\delta\delta^*+\delta^*\bar\delta=
\delta\bar\delta^*+\bar\delta^*\delta\,.$
\end{proof}

\begin{prop}\label{prop:bc-uguale-derham}
Let $(X,J,g,\omega)$ be a compact almost-K\"ahler manifold, then
$$
\mathcal{H}^{\bullet}_{BC(\delta,\bar\delta)}(X)=\mathcal{H}^{\bullet}_{\bar\delta}(X)\,.
$$
\end{prop}

\begin{proof}

Let $\alpha\in\mathcal{H}^k_{BC(\delta,\bar\delta)}(X)$; then by Lemma \ref{lemma:equivalence-bc-harmonic}, $\delta\alpha=0$,
$\bar\delta\alpha=0$ and $\delta^*\bar\delta^*\alpha=0$. We need to prove that $\bar\delta^*\alpha=0$. Using the almost-K\"ahler identities we have
$$
0=\delta^*\bar\delta^*\alpha=-i\delta^*[\delta,\Lambda]\alpha
$$
which means that $\delta^*\delta\Lambda\alpha=0$. Therefore, pairing with
$\Lambda\alpha$,
$$
0=(\delta^*\delta\Lambda\alpha,\Lambda\alpha)=
\vert\delta\Lambda\alpha\vert^2
$$
hence $\delta\Lambda\alpha=0$.
This, means that $\bar\delta^*\alpha=-i[\delta,\Lambda]\alpha=\delta\Lambda\alpha=0$, giving the first inclusion
$\mathcal{H}^{\bullet}_{BC(\delta,\bar\delta)}(X)\subseteq\mathcal{H}^{\bullet}_{\bar\delta}(X)$.\\
We now prove the other inclusion
$\mathcal{H}^{\bullet}_{\bar\delta}(X)\subseteq\mathcal{H}^{\bullet}_{BC(\delta,\bar\delta)}(X)$. Let $\alpha\in\mathcal{H}^k_{\bar\delta}(X)$, i.e., 
$\bar\delta\alpha=0$ and $\bar\delta^*\alpha=0$.
Moreover, since $\mathcal{H}^{\bullet}_{\bar\delta}(X)=
\mathcal{H}^{\bullet}_{\delta}(X)$ we also have that
$\delta\alpha=0$ and $\delta^*\alpha=0$. Hence, putting these relations together we have that $\bar\delta\alpha=0$, $\delta\alpha=0$ and $\delta^*\bar\delta^*\alpha=0$, i.e., by definition
$\alpha\in\mathcal{H}^{\bullet}_{BC(\delta,\bar\delta)}(X)$ giving the second inclusion.
\end{proof}

\begin{cor}
Let $(X,J,g,\omega)$ be a compact almost-K\"ahler manifold, then
$$
\mathcal{H}^{\bullet,\bullet}_{BC(\delta,\bar\delta)}(X)=\mathcal{H}^{\bullet,\bullet}_{d}(X)\,.
$$
\end{cor}

\begin{proof}
The thesis follows from the previous Proposition saying that $\mathcal{H}^{\bullet,\bullet}_{BC(\delta,\bar\delta)}(X)=\mathcal{H}^{\bullet,\bullet}_{\bar\delta}(X)$, 
the fact that $\mathcal{H}^{\bullet,\bullet}_{\bar\delta}(X)=\mathcal{H}^{\bullet,\bullet}_{\delta}(X)$ and
Theorem \ref{thm:equalities-harmonic-spaces}. 
\end{proof}

\begin{cor}
Let $(X,J,g,\omega)$ be a compact almost-K\"ahler manifold of dimension
$2n$, then
$$
\mathcal{H}^{\bullet}_{BC(\delta,\bar\delta)}(X)=\mathcal{H}^{\bullet}_{A(\delta,\bar\delta)}(X)\,.
$$
\end{cor}
\begin{proof}
First we show the inclusion $\mathcal{H}^{\bullet}_{BC(\delta,\bar\delta)}(X)\subseteq\mathcal{H}^{\bullet}_{A(\delta,\bar\delta)}(X)$.
Let $\alpha\in\mathcal{H}^{\bullet}_{BC(\delta,\bar\delta)}(X)$ then, by Lemma \ref{lemma:equivalence-bc-harmonic}, $\delta\alpha=0$, $\deltabar\alpha=0$ and $\delta\deltabar*\alpha=0$. Hence, $\delta\deltabar\alpha=0$ and 
by Propositions \ref{prop:bc-uguale-derham} and \ref{prop:deltabar-delta-laplacians},
 $\alpha\in\mathcal{H}^{\bullet}_{BC(\delta,\bar\delta)}(X)=\mathcal{H}^{\bullet}_{\bar\delta}(X)=\mathcal{H}^{\bullet}_{\delta}(X)$, so
 $\delta*\alpha=0$ and $\deltabar*\alpha=0$ giving the inclusion.\\
 The other inclusion follows from having $h^{\bullet}_{BC(\delta,\bar\delta)}(X)=h^{\bullet}_{A(\delta,\bar\delta)}(X)$. Indeed by using
 Remark \ref{rem:BC-hodge-duality}, Proposition \ref{prop:bc-uguale-derham} and Remark \ref{rem:delbar-hodge-duality} we have the following equalities on the dimensions, for any $k$,
$$
h^{k}_{A(\delta,\bar\delta)}(X)=h^{2n-k}_{BC(\delta,\bar\delta)}(X)=
h^{2n-k}_{\deltabar}(X)=h^{k}_{\deltabar}(X)=
h^{k}_{BC(\delta,\bar\delta)}(X)\,.
$$
\end{proof}

We prove the following Lemma

\begin{lemma}\label{lemma:commutator-laplacian-L}
Let $(X,J,g,\omega)$ be an almost-K\"ahler manifold, then
\begin{itemize}
\item[$\bullet$] $[L,\Delta_{\bar\delta}]=0$ and $[L,\Delta_{\delta}]=0$,
\item[$\bullet$] $[\Lambda,\Delta_{\bar\delta}]=0$ and $[\Lambda,\Delta_{\delta}]=0$.
\end{itemize}
\end{lemma}

\begin{proof}
We just need to prove the first equality
$$
[L,\Delta_{\bar\delta}]=[L,[\bar\delta,\bar\delta^*]]=
-[\bar\delta,[\bar\delta^*,L]]=i[\bar\delta,\delta]=0\,.
$$
\end{proof}

As a consequence we have the following Hard-Lefschetz Theorem
on the spaces of $\bar\delta-$ and $\Delta_{BC(\delta,\bar\delta)}-$harmonic forms.

\begin{theorem}\label{thm:hard-lefschetz}
Let $(X,J,g,\omega)$ be a compact almost-K\"ahler $2n$-dimensional manifold, then, for any $k$, the maps
$$
L^{k}:\mathcal{H}^{n-k}_{BC(\delta,\bar\delta)}(X)\to
\mathcal{H}^{n+k}_{BC(\delta,\bar\delta)}(X)
$$
are isomorphisms.
\end{theorem}

\begin{proof}
Since by the previous Lemma $[L,\Delta_{\bar\delta}]=0$ and $[\Lambda,\Delta_{\bar\delta}]=0$ and in general
$$
L^{k}:A^{n-k}(X)\to
A^{n+k}(X)
$$
are isomorphisms, then
the maps
$$
L^{k}:\mathcal{H}^{n-k}_{\bar\delta}(X)\to
\mathcal{H}^{n+k}_{\bar\delta}(X)
$$
are injective, and so isomorphisms by Remark \ref{rem:delbar-hodge-duality}.
The maps
$$
L^{k}:\mathcal{H}^{n-k}_{BC(\delta,\bar\delta)}(X)\to
\mathcal{H}^{n+k}_{BC(\delta,\bar\delta)}(X)
$$
are clearly isomorphic by Proposition \ref{prop:bc-uguale-derham}.
\end{proof}

For the bigraded case the result holds and it is proven in \cite[Theorem 5.1]{cirici-wilson-2}.

\section{Examples}

An important source of non-K\"ahler examples is furnished by nilmanifolds, namely compact quotients of a nilpotent connected simply-connected Lie group by a lattice.
On almost-complex nilmanifolds, given a left-invariant Hermitian metric, one can look for left-invariant harmonic (with respect to some operator) forms but in general they do not exhaust the whole space of harmonic forms. 
In the following we compute some examples showing that even on almost-K\"ahler manifolds we do not have a decomposition of the form
$$
\mathcal{H}^\bullet_{\bar\delta}(X)\neq\bigoplus_{p+q=\bullet} \mathcal{H}^{p,q}_{\bar\delta}(X)\,,
$$
differently from the case (cf. \cite[Theorem 4.1]{cirici-wilson-2})
$$
\mathcal{H}^\bullet_{\delbar}(X)\cap
\mathcal{H}^\bullet_{\mu}(X)=\bigoplus_{p+q=\bullet} \mathcal{H}^{p,q}_{\delbar}(X)\cap
\mathcal{H}^{p,q}_{\mu}(X)\,.
$$
Moreover, we will see that when the almost-Hermitian structure is not almost-K\"ahler the equalities in Theorem \ref{thm:equalities-harmonic-spaces} may fail and
also the inequalities in Corollary \ref{cor:comparison-betti-numbers} may fail,
in particular we construct an example where
$$
h^2_{\deltabar}(X)>b_2(X)\,.
$$

\begin{ex}\label{example-1}

Let $\mathbb{H}(3;\mathbb{R})$ be the $3$-dimensional Heisenberg group and
$\mathbb{H}(3;\mathbb{Z})$ be the subgroup of matrices with entries in $\mathbb{Z}$.
The Kodaira-Thurston manifold is defined as the quotient
$$
X:=\left(\mathbb{H}(3;\mathbb{R})\times \mathbb{R}\right)/
\left(\mathbb{H}(3;\mathbb{Z})\times \mathbb{Z}\right)\,.
$$
The manifold $X$ is a $4$-dimensional nilmanifold which admits both complex and symplectic structures.
We consider the non-integrable almost-complex structure $J$ defined by the structure equations
$$
\left\lbrace
\begin{array}{lcl}
d\varphi^1 & =& 0\\
d\varphi^2 &=& \frac{1}{2i}\varphi^{12}+\frac{1}{2i}\left(\varphi^{1\bar 2}-\varphi^{2\bar 1}\right)+\frac{1}{2i}\varphi^{\bar 1\bar 2}
\end{array}
\right.\,
$$
where $\left\lbrace\varphi^1\,,\varphi^2\right\rbrace$ is a global co-frame of (1,0)-forms on $X$.
The $(1,1)$-form $\omega:=\frac{1}{2i}\left(\varphi^{1\bar 1}+\varphi^{2\bar 2}\right)$ is a compatible symplectic structure, hence the pair $(J,\omega)$ induces a almost-K\"ahler structure on $X$.\\
Recall that $\mathcal{H}^{\bullet,\bullet}_{\bar\delta}(X)=\mathcal{H}^{\bullet,\bullet}_{d}(X)$, but in general
 $\mathcal{H}^{\bullet}_{\bar\delta}(X)\subseteq\mathcal{H}^{\bullet}_{d}(X)$.
 One can easily compute the spaces of left-invariant harmonic forms and one gets
$$
\displaystyle\begin{array}{lcl}
\mathcal{H}^1_{\bar\delta,\text{inv}}(X) & = &  \displaystyle
\left\langle \varphi^1,\,\bar\varphi^1 \right\rangle, \\[5pt]
\mathcal{H}^2_{\bar\delta,\text{inv}}(X) & = &  \displaystyle
\left\langle \varphi^{1\bar 1}\,,\varphi^{2\bar 2}\,,
\varphi^{12}-\varphi^{\bar 1\bar 2}\,,
\varphi^{1\bar 2}+\varphi^{2\bar 1}
 \right\rangle
\end{array}
$$
and (cf. also \cite{cirici-wilson-2})
$$
\displaystyle\begin{array}{lcl}
\mathcal{H}^{1,0}_{\bar\delta,\text{inv}}(X) & = &  \displaystyle
\left\langle \varphi^1 \right\rangle, \\[5pt]
\mathcal{H}^{0,1}_{\bar\delta,\text{inv}}(X) & = &  \displaystyle
\left\langle \bar\varphi^1 \right\rangle, \\[5pt]
\mathcal{H}^{2,0}_{\bar\delta,\text{inv}}(X) & = &  \displaystyle
0\,,\\[5pt]
\mathcal{H}^{0,2}_{\bar\delta,\text{inv}}(X) & = &  \displaystyle
0\,,\\[5pt]
\mathcal{H}^{1,1}_{\bar\delta,\text{inv}}(X) & = &  \displaystyle
\left\langle \varphi^{1\bar 1}\,,\varphi^{2\bar 2}\,,
\varphi^{1\bar 2}+\varphi^{2\bar 1}
 \right\rangle
\end{array}
$$
The remaining spaces can be computed easily by duality.
A first observation is that
$$
\mathcal{H}^2_{\bar\delta,\text{inv}}(X)\neq\bigoplus_{p+q=2} \mathcal{H}^{p,q}_{\bar\delta,\text{inv}}(X)\,. 
$$
In particular, by \cite[Theorem 4.1]{cirici-wilson-2} and Theorem
\ref{thm:equalities-harmonic-spaces}
$$
\mathcal{H}^2_{\delbar,\text{inv}}(X)\cap
\mathcal{H}^2_{\mu,\text{inv}}(X)=
\bigoplus_{p+q=2}\mathcal{H}^{p,q}_{\delbar,\text{inv}}(X)\cap
\mathcal{H}^{p,q}_{\mu,\text{inv}}(X)=
\bigoplus_{p+q=2}\mathcal{H}^{p,q}_{\bar\delta,\text{inv}}(X)\subset
\mathcal{H}^2_{\bar\delta,\text{inv}}(X)\,,
$$
Therefore, also in the almost-K\"ahler case we can have (cf. Proposition
\ref{prop:delta-delbar-delta-deltabar})
$$
\mathcal{H}^2_{\delbar,\text{inv}}(X)\cap
\mathcal{H}^2_{\mu,\text{inv}}(X)\neq
\mathcal{H}^2_{\bar\delta,\text{inv}}(X).
$$
Moreover, 
since $\text{dim}\,\mathcal{H}^2_{\bar\delta,\text{inv}}(X)=4=b_2(X)$,
from Corollary \ref{cor:comparison-betti-numbers} we have that
$$
\mathcal{H}^2_{\bar\delta,\text{inv}}(X)=\mathcal{H}^2_{\bar\delta}(X)\,.
$$
Since $\text{dim}\,\mathcal{H}^1_{\bar\delta,\text{inv}}(X)=2$
then
$$
2\leq\text{dim}\,\mathcal{H}^1_{\bar\delta}(X)\leq b_1(X)=3\,,
$$
hence $\text{dim}\,\mathcal{H}^1_{\bar\delta}(X)=3$ if and only if
there exists a non-left-invariant $\bar\delta$-harmonic $1$-form\,.
\end{ex}

\begin{ex}\label{example-2}
Let $X$ be the $4$-dimensional Filiform nilmanifold and consider the non-integrable almost-complex structure $J$ defined by the following structure equations
$$
\left\lbrace
\begin{array}{lcl}
d\varphi^1 & =& 0\\
d\varphi^2 &=& \frac{1}{2i}\varphi^{12}+\frac{1}{2i}\left(\varphi^{1\bar 2}-\varphi^{2\bar 1}\right)-i\varphi^{1\bar 1}+\frac{1}{2i}\varphi^{\bar 1\bar 2}
\end{array}
\right.
$$
where $\left\lbrace\varphi^1\,,\varphi^2\right\rbrace$ is a global co-frame of (1,0)-forms on $X$.
As observed in \cite{cirici-wilson-2} $J$ does not admit any compatible symplectic structure.
We fix the diagonal metric $\omega:=\frac{1}{2i}\left(\varphi^{1\bar 1}+\varphi^{2\bar 2}\right)$. 
 One can easily compute the spaces of left-invariant harmonic forms and one gets
$$
\displaystyle\begin{array}{lcl}
\mathcal{H}^1_{\bar\delta,\text{inv}}(X) & = &  \displaystyle
\left\langle \varphi^1,\,\bar\varphi^1 \right\rangle, \\[5pt]
\mathcal{H}^2_{\bar\delta,\text{inv}}(X) & = &  \displaystyle
\left\langle 
\varphi^{12}-\varphi^{\bar 1\bar 2}\,,
-\frac{1}{2}\varphi^{1\bar 1}+\varphi^{1\bar 2}-\frac{1}{2}\varphi^{2\bar 2}\,,
\frac{1}{2}\varphi^{1\bar 1}+\varphi^{2\bar 1}+\frac{1}{2}\varphi^{2\bar 2}
 \right\rangle
\end{array}
$$
and 
$$
\displaystyle\begin{array}{lcl}
\mathcal{H}^{1,0}_{\bar\delta,\text{inv}}(X) & = &  \displaystyle
\left\langle \varphi^1 \right\rangle, \\[5pt]
\mathcal{H}^{0,1}_{\bar\delta,\text{inv}}(X) & = &  \displaystyle
\left\langle \bar\varphi^1 \right\rangle, \\[5pt]
\mathcal{H}^{2,0}_{\bar\delta,\text{inv}}(X) & = &  \displaystyle
0\,,\\[5pt]
\mathcal{H}^{0,2}_{\bar\delta,\text{inv}}(X) & = &  \displaystyle
0\,,\\[5pt]
\mathcal{H}^{1,1}_{\bar\delta,\text{inv}}(X) & = &  \displaystyle
\left\langle -\frac{1}{2}\varphi^{1\bar 1}+\varphi^{1\bar 2}-\frac{1}{2}\varphi^{2\bar 2}\,,
\frac{1}{2}\varphi^{1\bar 1}+\varphi^{2\bar 1}+\frac{1}{2}\varphi^{2\bar 2}
 \right\rangle
\end{array}
$$
The remaining spaces can be computed easily by duality.
Since $(J,\omega)$ is not an almost-K\"ahler structure we cannot apply Corollary
\ref{cor:comparison-betti-numbers}, in particular in this case we have the opposite inequality
$$
\text{dim}\,\mathcal{H}^2_{\bar\delta}(X)\geq
3=\text{dim}\,\mathcal{H}^2_{\bar\delta,\text{inv}}(X)>2=b_2(X)\,.
$$
Moreover, one can easily compute the $\Delta_{BC(\delta,\bar\delta)}$-harmonic forms
$$
\displaystyle\begin{array}{lcl}
\mathcal{H}^1_{BC(\delta,\bar\delta),\text{inv}}(X) & = &  \displaystyle
\left\langle \varphi^1,\,\bar\varphi^1 \right\rangle, \\[5pt]
\mathcal{H}^2_{BC(\delta,\bar\delta),\text{inv}}(X) & = &  \displaystyle
\left\langle 
\varphi^{1\bar 1}\,,
\varphi^{12}-\varphi^{\bar 1\bar 2}\,,
\varphi^{1\bar 2}-\frac{1}{2}\varphi^{2\bar 2}\,,
\varphi^{2\bar 1}+\frac{1}{2}\varphi^{2\bar 2}
 \right\rangle
\end{array}
$$
and 
$$
\displaystyle\begin{array}{lcl}
\mathcal{H}^{1,0}_{BC(\delta,\bar\delta),\text{inv}}(X) & = &  \displaystyle
\left\langle \varphi^1 \right\rangle, \\[5pt]
\mathcal{H}^{0,1}_{BC(\delta,\bar\delta),\text{inv}}(X) & = &  \displaystyle
\left\langle \bar\varphi^1 \right\rangle, \\[5pt]
\mathcal{H}^{2,0}_{BC(\delta,\bar\delta),\text{inv}}(X) & = &  \displaystyle
0\,,\\[5pt]
\mathcal{H}^{0,2}_{BC(\delta,\bar\delta),\text{inv}}(X) & = &  \displaystyle
0\,,\\[5pt]
\mathcal{H}^{1,1}_{BC(\delta,\bar\delta),\text{inv}}(X) & = &  \displaystyle
\left\langle \varphi^{1\bar 1}\,,
\varphi^{1\bar 2}-\frac{1}{2}\varphi^{2\bar 2}\,,
\varphi^{2\bar 1}+\frac{1}{2}\varphi^{2\bar 2}
 \right\rangle
\end{array}
$$
Notice that, unlike the almost-K\"ahler case
$$
\mathcal{H}^2_{BC(\delta,\bar\delta),\text{inv}}(X)\neq
\mathcal{H}^2_{\bar\delta,\text{inv}}(X)
$$
and
$$
\mathcal{H}^{1,1}_{BC(\delta,\bar\delta),\text{inv}}(X)\neq
\mathcal{H}^{1,1}_{\bar\delta,\text{inv}}(X)\,.
$$
\end{ex}

\begin{ex}
Let $X:=\mathbb{I}_3$ be the Iwasawa manifold, namely the quotient of the complex $3$-dimensional Heisenberg group $\mathbb{H}(3;\mathbb{C})$ by the subgroup of matrices with entries in $\mathbb{Z}[i]$. The manifold $X$ is a $6$-dimensional nilmanifold admitting both complex and symplectic structures.
Then there exists a global co-frame of $1$-forms $\left\lbrace e^i\right\rbrace_{i=1,\,\cdots\,,6}$ satisfying the following structure equations
$$
\left\lbrace
\begin{array}{lcl}
d\,e^1 & =& 0\\
d\,e^2 & =& 0\\
d\,e^3 & =& 0\\
d\,e^4 & =& 0\\
d\,e^5 & =& -e^{13}+e^{24}\\
d\,e^6 & =& -e^{14}-e^{23}\,.
\end{array}
\right.
$$
We define the following non-integrable almost-complex structure
$$
Je^1=-e^6,\quad Je^2=-e^5, \quad Je^3=-e^4
$$
and consider the compatible symplectic structure
$$
\omega:=e^{16}+e^{25}+e^{34}.
$$
Therefore, $(X,J,\omega)$ is a compact $6$-dimensional almost-K\"ahler manifold.
We set
$$
\left\lbrace
\begin{array}{lcl}
\varphi^1 & =& e^1+ie^6\\
\varphi^2 & =& e^2+ie^5\\
\varphi^3 & =& e^3+ie^4
\end{array}
\right.
$$
then the structure equations become
$$
\left\lbrace
\begin{array}{lcl}
d\,\varphi^1 & =& \left(-\frac{1}{4}\varphi^{13}-\frac{i}{4}\varphi^{23}\right)+\left(\frac{1}{4}\varphi^{1\bar3}+\frac{1}{4}\varphi^{3\bar 1}-
\frac{i}{4}\varphi^{2\bar3}+\frac{i}{4}\varphi^{3\bar2}\right)+
\left(\frac{1}{4}\varphi^{\bar1\bar3}-\frac{i}{4}\varphi^{\bar2\bar3}\right)\\
d\,\varphi^2 & =& \left(-\frac{i}{4}\varphi^{13}+\frac{1}{4}\varphi^{23}\right)+\left(-\frac{i}{4}\varphi^{1\bar3}+\frac{i}{4}\varphi^{3\bar 1}-
\frac{1}{4}\varphi^{2\bar3}-\frac{1}{4}\varphi^{3\bar2}\right)+
\left(-\frac{i}{4}\varphi^{\bar1\bar3}-\frac{1}{4}\varphi^{\bar2\bar3}\right)\\
d\,\varphi^3 & =& 0
\end{array}
\right.
$$
One can  compute the spaces of left-invariant harmonic forms and one gets
$$
\displaystyle\begin{array}{lcl}
\mathcal{H}^1_{\bar\delta,\text{inv}}(X) & = &  \displaystyle
\left\langle \varphi^3,\,\bar\varphi^3 \right\rangle, \\[5pt]
\mathcal{H}^2_{\bar\delta,\text{inv}}(X) & = &  \displaystyle
\left\langle \varphi^{1\bar 1}+\varphi^{2\bar 2}\,,
\varphi^{3\bar3}\,, -\varphi^{12}+\varphi^{\bar1\bar2}
 \right\rangle,\\[5pt]
 \mathcal{H}^3_{\bar\delta,\text{inv}}(X) & = &  \displaystyle
\left\langle \varphi^{13\bar 1}+\varphi^{23\bar 2}\,,
\varphi^{1\bar1\bar3}+\varphi^{2\bar2\bar3}
 \right\rangle
\end{array}
$$
and 
$$
\displaystyle\begin{array}{lcl}
\mathcal{H}^{1,0}_{\bar\delta,\text{inv}}(X) & = &  \displaystyle
\left\langle \varphi^3 \right\rangle, \\[5pt]
\mathcal{H}^{0,1}_{\bar\delta,\text{inv}}(X) & = &  \displaystyle
\left\langle \bar\varphi^3 \right\rangle, \\[5pt]
\mathcal{H}^{2,0}_{\bar\delta,\text{inv}}(X) & = &  \displaystyle
0\,,\\[5pt]
\mathcal{H}^{0,2}_{\bar\delta,\text{inv}}(X) & = &  \displaystyle
0\,,\\[5pt]
\mathcal{H}^{1,1}_{\bar\delta,\text{inv}}(X) & = &  \displaystyle
\left\langle  \varphi^{1\bar 1}+\varphi^{2\bar 2}\,,
\varphi^{3\bar3}\right\rangle\,,\\[5pt]
\mathcal{H}^{2,1}_{\bar\delta,\text{inv}}(X) & = &  \displaystyle
\left\langle\varphi^{13\bar1}+\varphi^{23\bar 2}\right\rangle\,,\\[5pt]
\mathcal{H}^{1,2}_{\bar\delta,\text{inv}}(X) & = &  \displaystyle
\left\langle
\varphi^{1\bar1\bar3}+\varphi^{2\bar2\bar3}
\right\rangle\,.
\end{array}
$$
The remaining spaces can be computed easily by duality.
In particular
$$
\mathcal{H}^2_{\bar\delta,\text{inv}}(X)\neq\bigoplus_{p+q=2} \mathcal{H}^{p,q}_{\bar\delta,\text{inv}}(X)\,. 
$$
\end{ex}

\end{document}